\documentclass[10pt]{amsart}
\usepackage{microtype}

\usepackage{latexsym,exscale,enumerate,amsfonts,amssymb, ulem, xparse, mathtools}
\usepackage{amsmath,amsthm,amsfonts,amssymb,amscd, stmaryrd,textcomp}
\usepackage{hyperref}
\usepackage{ulem}
\usepackage{bbold}
\usepackage[usenames,dvipsnames]{xcolor}
\colorlet{green}{black!30!green} 

\usepackage[all]{xy}
\SelectTips{cm}{}

\usepackage{graphicx}

\usepackage{tikz}
\usetikzlibrary{calc}
\usetikzlibrary{decorations.markings}
\usetikzlibrary{decorations.pathreplacing}
\usetikzlibrary{arrows,shapes,positioning}
\tikzstyle directed=[postaction={decorate,decoration={markings,
    mark=at position #1 with {\arrow{>}}}}]
\tikzstyle rdirected=[postaction={decorate,decoration={markings,
    mark=at position #1 with {\arrow{<}}}}]
\tikzset{anchorbase/.style={baseline={([yshift=-0.5ex]current bounding box.center)}}}

\theoremstyle{theorem}
\newtheorem{theorem}{Theorem}

\newtheorem{corollary}[theorem]{Corollary}

\newtheorem{lemma}[theorem]{Lemma}
\newtheorem{proposition}[theorem]{Proposition}
\theoremstyle{definition}
\newtheorem{definition}[theorem]{Definition}
\newtheorem{example}[theorem]{Example}

\newtheorem{remark}[theorem]{Remark}

\newcommand{\slnn}[1]{\mf{sl}_{#1}}
\newcommand{\glnn}[1]{\mf{gl}_{#1}}

\newcommand{\diagrep}{\phi} 
\newcommand{\pTr}{\mathrm{pTr}}
\newcommand{\TL}{\cat{TL}}
\newcommand{\ATL}{\cat{A}\cat{TL}}\newcommand{\essATL}{\cat{A}\cat{TL}^{\mathrm{ess}}}
\newcommand{\sessATL}{\cat{A}\cat{TL}_\mathrm{s}^{\mathrm{ess}}}
\newcommand{\wrap}{D}
\newcommand{\wrapi}{D^{-1}}

\newcommand{\reprs}{\mathrm{K_0(Rep(\slnn{2}))}}
\newcommand{\rep}{\mathrm{Rep(\slnn{2})}}
\newcommand{\reph}{\mathrm{Rep(\h)}}
\newcommand{\bVn}{\Lambda}
\newcommand{\sym}{\mathrm{Sym}}
\newcommand{\Kar}{\operatorname{Kar}}
\newcommand{\Cu}{\mathrm{Cu}}
\newcommand{\Ca}{\mathrm{Ca}}

\newcommand{\cat}[1]{\ensuremath{\mbox{\bfseries {\upshape {#1}}}}}

\DeclareMathOperator{\Hom}{Hom}
\DeclareMathOperator{\Sym}{Sym}

\newcommand{\id}{{\rm id}}
\def\mf{\mathfrak}

\def\comm#1{}%

\def\emph#1{{\sl #1\/}}

\let\phi=\varphi
\let\theta=\vartheta
\let\epsilon=\varepsilon

\usepackage{bbm}
\def\C{{\mathbbm C}}
\def\N{{\mathbbm N}}
\def\R{{\mathbbm R}}
\def\Z{{\mathbbm Z}}
\def\Q{{\mathbbm Q}}

\def\k{{\mathbbm k}}
\def\h{{\mathfrak h}}

\def\Link{\mathrm{Link}}
\def\Sk{\mathrm{Sk}}
\newcommand{\Ss}{\mathbb{S}}

\def\A{\mathrm{A}}
\def\Kh{\mathrm{Kh}}
\def\T{\mathrm{T}}

\def\la{\langle}
\def\ra{\rangle}

\begin{document}
%

\title
{
Extremal weight projectors
}

\author{Hoel Queffelec}
\address{CNRS and IMAG, Universit\'e de Montpellier, France}
\email{hoel.queffelec@umontpellier.fr}

 \author{Paul Wedrich}
 \address{Imperial College London, United Kingdom}
 \email{p.wedrich@gmail.com}

\thanks{Funding: H.~Q. was partially supported by the ANR Quantact and a grant from the Universit\'e de Montpellier. P.~W. was supported by a Fortbildungsstipendium at the Max Planck Institute for Mathematics and by the Leverhulme Research Grant RP2013-K-017 to Dorothy Buck, Imperial College London.}

\begin{abstract} We introduce a quotient of the affine Temperley-Lieb category that encodes all weight-preserving linear maps between finite-dimensional $\slnn{2}$-representations. We study the diagrammatic idempotents that correspond to projections onto extremal weight spaces and find that they  satisfy similar properties as Jones-Wenzl projectors, and that they categorify the Chebyshev polynomials of the first kind. This gives a categorification of the Kauffman bracket skein algebra of the annulus, which is well adapted to the task of categorifying the multiplication on the Kauffman bracket skein module of the torus. 
\end{abstract}

\maketitle

\section{Introduction}
The Lie algebra $\slnn{2}$ and its universal enveloping algebra $U(\slnn{2})$ are ubiquitous in representation theory, topology and physics, and their study typically starts with the description of the monoidal category $\rep$ of finite-dimensional $\slnn{2}$-representations. The \emph{representation ring} of $\slnn{2}$, that is, the Grothendieck group $\reprs$ with multiplication inherited from the tensor product of representations, is isomorphic to the polynomial ring $\Z[X]$ generated by the class of the vector representation $X=[V]$.

Two bases of $\Z[X]$, triangularly equivalent (over $\Q$) to the monomial basis, play a prominent role, especially in quantum topological applications:
\begin{itemize}
\item the Chebyshev polynomials of the first kind $L_m\in \Z[X]$, recursively defined as $L_0=2$, $L_1=X$ and $L_{m}= X L_{m-1} - L_{m-2}$ for $m\geq 3$;
\item the Chebyshev polynomials of the second kind $J_m\in \Z[X]$, defined to satisfy the same recurrence, but with initial conditions $J_0=1$, $J_1=X$.
\end{itemize}

\noindent
The basis elements $J_m$ appear in $\reprs$ as the classes of the simple representations $\Sym^m(V)$. 
Since tensor powers of $V$ can be decomposed into simple representations, this implies that the basis change in $\Z[X]$ from the monomial basis to $\{J_m\}$ has positive coefficients.

It is easy to see that the basis change from $\{J_m\}$ to $\{L_m\}$ also has positive coefficients, and in fact $J_m$ = $L_m + J_{m-2}$ for $m\geq 2$. The element $L_m$ of the representation ring, thus, corresponds to the formal difference of the classes of $\Sym^m(V)$ and  $\Sym^{m-2}(V)$, which we interpret as the class of the \emph{extremal weight space} of $\Sym^m(V)$, that is, the direct sum of the highest and lowest weight spaces. Another indication that the basis $\{L_m\}$ is more fundamental than the basis $\{ J_m\}$ is given by comparing the multiplication rules:
\begin{equation} \label{eq:mult}
 J_m J_n = J_{m+n} + J_{m+n-2}+\cdots +J_{|m-n|} \quad\text{and}\quad L_m L_n = L_{m+n} + L_{|m-n|}.
\end{equation}

The representation category $\rep$ admits a diagrammatic generators-and-relations description as the Karoubi envelope of the famous Temperley-Lieb category $\TL$, which encodes intertwiners between tensor powers of the vector representation, see Section~\ref{sec:TL}. In this description, the simple representations $\Sym^m(V)$ are obtained through the Jones-Wenzl projectors $P_m\in \TL$, which correspond to the idempotent intertwiners:
\[
V^{\otimes m} \twoheadrightarrow \Sym^m(V)\hookrightarrow V^{\otimes m}.
\]
In summary, the Karoubi envelope of the Temperley-Lieb category gives a diagrammatic categorification of the polynomial ring and the Jones-Wenzl projectors $P_m$ categorify the polynomials $J_m$. The purpose of this paper is to give a description of analogs of the Jones-Wenzl projectors for the Chebyshev polynomials of the first kind $L_m$, which we call \emph{extremal weight projectors}.

\subsection{Main results}
We define a diagrammatic extension $\essATL$ of the Temperley-Lieb category, which utilizes diagrams in an annulus. More precisely, the monoidal category $\essATL$ is the quotient of the affine Temperley-Lieb category $\ATL$ by an ideal generated by the relation that an essential circle in the annulus is zero, see Definition~\ref{def:ATL} for details. 
In this category, we give a very simple recursive diagrammatic definition of morphisms $T_m$ for $m\geq 0$, which is reminiscent of, but simpler than the recursive definition of the Jones-Wenzl projectors. Our main results are:
\begin{itemize}
\item  $\essATL$ describes the category of weight-preserving linear maps between $\slnn{2}$-representations, or in other words, the representation category $\reph\supset\rep$ of the Cartan $\h\subset \slnn{2}$; 
\item the elements $T_m$ are idempotents, corresponding to projections onto extremal weight spaces;
\item in the Karoubi envelope of $\essATL$, there are isomorphisms $T_m\otimes T_n\cong T_{m+n}\oplus T_{|m-n|}$, which categorify the multiplication rule in \eqref{eq:mult}.
 \end{itemize}
The simple definition of these idempotents makes them natural objects to study. Additionally, we expect that they will play a central role in categorifying skein algebras, as we will explain next.

\subsection{Skein modules and skein algebras} Let $M$ be an orientable 3-manifold and $\k:=\Z[q^{\pm 1/2}]$. Following Przytycki \cite{Prz1} and Turaev \cite{Tur}, the Kauffman bracket skein module $\Sk_q(M)$ of $M$ is constructed by forming the free $\k$-module spanned by isotopy classes of framed links in $M$ and taking the quotient by the Kauffman bracket skein relations \cite{KauffBracket} supported in 3-balls in $M$:
\begin{equation}
\begin{tikzpicture}[anchorbase, scale=.5]
\draw [very thick] (1,0) to [out=90,in=270] (0,1.5);
\draw [white,line width=.15cm] (0,0) to [out=90,in=270] (1,1.5);
\draw [very thick] (0,0) to [out=90,in=270] (1,1.5);
\end{tikzpicture}
\;= \; q^{1/2}\;
\begin{tikzpicture}[anchorbase, scale=.5]
\draw [very thick] (0,0) -- (0,1.5);
\draw [very thick] (1,0) -- (1,1.5);
\end{tikzpicture}
\;
+ \; q^{-1/2}\;
\begin{tikzpicture}[anchorbase, scale=.5]
\draw [very thick] (0,0) to [out=90,in=180] (.5,.5) to [out=0,in=90] (1,0);
\draw [very thick] (0,1.5) to [out=270,in=180] (.5,1) to [out=0,in=270] (1,1.5);
\end{tikzpicture} 
\quad,\quad
\begin{tikzpicture}[anchorbase, scale=.5]
\draw [very thick] (0,0) to [out=90,in=180] (.5,.5) to [out=0,in=90] (1,0)to [out=270,in=0] (.5,-.5) to [out=180,in=270] (0,0);
\end{tikzpicture} 
\;= \; -(q +q^{-1})
\end{equation}
For manifolds with boundary, one may take tangles with endpoints in a specified discrete set of boundary points as spanning set for the skein module. Thus, the categories $\TL$ and $\essATL$ can be seen as being built from skein modules of thickened squares or thickened annuli with prescribed boundary points, at $q=1$. The multiplication is induced from gluing these 3-manifolds along the vertical part of their boundary.

We write $\Sk_q(\Sigma)$ for the skein modules of thickened surfaces $\Sigma \times [0,1]$. These become \emph{skein algebras} with the multiplication operations obtained by stacking two copies of the thickened surface i.e. gluing along the horizontal part of their boundary. The skein algebra $\Sk_q(\R^3)$ is free of rank one over $\k$ and the class of a framed link is given by its Kauffman bracket evaluation, which recovers the Jones polynomial.  
The skein algebra $\Sk_q(\A)$ of the annulus without boundary points is isomorphic to the polynomial ring $\k[X]$ generated by the core curve of the annulus, and thus to the representation ring $\reprs\otimes \k$. Elements of the skein of the annulus act on framed knots in skein modules by cabling operations. For example, cabling knots in $\Sk_q(\R^3)$ by the element corresponding to $J_m\in \Z[X]$ produces the $m$-th colored Jones polynomial as evaluation.

\subsection{Positive bases for skein algebras and categorification}
The Kauffman bracket skein module of the torus $\T=\Ss^1\times \Ss^1$ has a basis given by collections of parallel non-separating simple closed curves on $\T$. Thus, as a $\k$-module, it is isomorphic to a direct sum of countably many copies of $\Sk_q(\A)$, one for each slope $m/n$ on the torus. However, Frohman and Gelca \cite{FG} proved that the skein multiplication has a particular nice form in terms of the basis assembled from the Chebyshev polynomials $L_d$ in the copies $\Sk_q(\A)$. In fact, for a pair $(r,s)\in \Z^2$, they define $(r,s)_T$ to be the basis element $L_{\gcd(r,s)}$ of slope $r/s$ in $\Sk_q(\T)$ and show:
\begin{equation}\label{eq:FG}
(m,n)_T*(r,s)_T = q^{(m s-r n)/2} (m+r,n+s)_T+ q^{(r n-m s)/2} (m-r,n-s)_T 
\end{equation}

This formula, taken at roots of unity, plays a central role in Bonahon and Wong's \emph{miraculous cancellations} \cite{BW1} and the related work of Le \cite{Le16b}, in connection to cluster algebra. Also in relation with cluster algebras, for arbitrary closed oriented surfaces $\Sigma$, Thurston showed that the elements $L_d\in \Sk(\A)$, embedded along simple closed curves in $\Sigma$, give a basis for the skein algebra $\Sk_{q=1}(\Sigma)$ with positive structure constants, \cite{Thu}. He also conjectured that the same holds for generic $q$ (see also \cite{Le16a}) and asked whether this positivity phenomenon might be the shadow of a monoidal categorification of $\Sk_q(\Sigma)$. We propose a definition.
\begin{definition}\label{def:moncat} A \emph{monoidal categorification} of $\Sk_q(\Sigma)$ consists of a graded, monoidal, additive category $\mathcal{C}(\Sigma)$ and an isomorphism of unital $\k$-algebras $\psi\colon\Sk_q(\Sigma)\cong K_0(\mathcal{C}(\Sigma))$. 
\end{definition}
Thurston's positivity conjecture suggests that one should look for monoidal categorifications of $\Sk_q(\Sigma)$ with simple objects categorifying the cablings of simple closed curves by $L_d$. This is the main motivation for our construction of the highest weight projectors. Indeed, in Section~\ref{sec:decat} we describe a subcategory of $\essATL$, whose Karoubi envelope gives a monoidal categorification of $\Sk_q(\A)$ (after introducing an extra grading), and the idempotents $T_d$ categorify the elements $L_d$.

A monoidal categorification of a skein algebra becomes even more interesting if one can evaluate links in it. To make this precise, let $\Link(\Sigma)$ denote the monoidal category whose objects are given by embedded links $L$ in $\Sigma\times [0,1]$ and with morphisms given by compact, oriented bordisms in $\Sigma\times [0,1]\times [0,1]$ between links, modulo isotopy. The composition is the usual gluing of bordisms and the tensor product is induced by stacking links (and bordisms) along the thickening direction and rescaling: $[0,2]\to [0,1]$. We now define a categorification of a skein algebra to consist of a monoidal categorification as in Definition~\ref{def:moncat} together with a Khovanov-type link homology functor (\cite{Kh1}) that is compatible with the monoidal structure.

\begin{definition} A \emph{categorification of the skein algebra} $\Sk_q(\Sigma)$ consists of  a monoidal categorification as in Definition~\ref{def:moncat}, together with a monoidal functor $\Kh\colon \Link(\Sigma) \to \mathcal{C}(\Sigma)$, such that the class of each link $L\subset \Sigma \times [0,1]$ in $\Sk_q(\Sigma)$ is sent to $K_0(\Kh(L))$ under $\psi$.
\end{definition}

Recall that Bar-Natan's construction of Khovanov homology admits an extension to the case of links in $\Sigma\times [0,1]$, see \cite{BN2}[Section 11.6]. The target category $\mathcal{C}(\Sigma)$ of this construction is the homotopy category of chain complexes with chain groups generated by $1$-manifolds embedded in $\Sigma$ and differentials generated by cobordisms in $\Sigma \times [0,1]$, modulo certain relations. This gives a categorification of the skein module $\Sk_{-q}(\Sigma)$ and a proposal for a (projective) Khovanov functor $\Kh\colon \Link(\Sigma) \to \mathcal{C}(\Sigma)$ (see also \cite{APS} for a related construction). However, it is not known whether $\mathcal{C}(\Sigma)$ can be equipped with a monoidal structure making $\Kh$ monoidal.

In \cite{QW} we will examine this problem in the case of the torus $\Sigma=\T$. In this case, Bar-Natan's cobordism category contains what we call \emph{slope subcategories} --- subcategories of $1$-manifolds and cobordisms which are $\Ss^1$-equivariant along a fixed slope. Taking the quotient by the $\Ss^1$-action, it is not hard to see that the degree zero part of this subcategory is closely related to the affine Temperley-Lieb category $\ATL$. The additional relation in $\essATL$ is realized by setting the endo-cobordism of $\emptyset$ that is given by the boundary parallel torus equal to zero. The homotopy category over a corresponding quotient cobordism category, thus, contains idempotents which categorify the Frohman-Gelca basis elements $(m,n)_T$. Moreover, it follows from this translation and our results in Section~\ref{sec:parprod} that these idempotents satisfy a categorified version of the Frohman-Gelca formula \eqref{eq:FG} in the case when they lie in the same slope. In \cite{QW} we treat the general case of tensoring such categorified Frohman-Gelca basis elements, although we have to work in a slightly different setting to avoid functoriality problems. We also speculate that a similar construction should be possible for other closed surfaces $\Sigma$.

 \vspace{.5cm}

 \noindent {\bf Acknowledgements:} We would like to thank 
Nils Carqueville, 
Charlie Frohman,
Eugene Gorsky,
Mikhail Khovanov,
Scott Morrison,
Jake Rasmussen,
Philippe Roche,
Heather Russell
and Catharina Stroppel
for interesting discussions and helpful comments.


\section{Affine Temperley-Lieb and extremal weight projectors}

\subsection{Temperley-Lieb categories}
\label{sec:TL}
We start by recalling classical results about the Temperley-Lieb category and its affine version at $q=1$. 

\begin{definition} Let $\TL(m,n)$ denote the free $\C$-module spanned by planar matchings in a box with $m$ boundary points on the bottom and $n$ boundary points on the top (and no closed components). The Temperley-Lieb category $\TL$ is the $\C$-linear strict monoidal category with objects $n\in \N$, morphism spaces $\TL(n,m)$ and multiplication $\TL(m,n)\times \TL(l,m) \to \TL(l,n)$ is defined on the bases of matchings by composing planar tangles and reducing each resulting closed components to $-2$. The monoidal structure is defined on objects by $m\otimes n=m+n$ and on morphisms by placing diagrams side by side.
\end{definition}

\begin{example} We give examples for the composition and the tensor product of two planar matchings. 
\[
\begin{tikzpicture} [scale=.4,anchorbase]
\draw (-1.5,0) rectangle (3.5,1.5);
\draw [very thick] (0,0) to [out=90,in=180] (.5,.5) to [out=0,in=90] (1,0);
\draw [very thick] (0,1.5) to [out=270,in=180] (.5,1) to [out=0,in=270] (1,1.5);
\draw [very thick] (-1,0) to [out=90,in=270] (2,1.5);
\draw [very thick] (2,0) to [out=90,in=180] (2.5,.5) to [out=0,in=90] (3,0);
\end{tikzpicture}
\cdot \begin{tikzpicture} [scale=.4,anchorbase]
\draw (-1.5,0) rectangle (3.5,1.5);
\draw [very thick] (-1,1.5) to [out=270,in=180] (.5,.5) to [out=0,in=270] (2,1.5);
\draw [very thick] (0,1.5) to [out=270,in=180] (.5,1) to [out=0,in=270] (1,1.5);
\draw [very thick] (1.5,0) to [out=90,in=270] (3,1.5);
\end{tikzpicture} 
=
-2\; \begin{tikzpicture} [scale=.4,anchorbase]
\draw [very thick] (0,1.5) to [out=270,in=180] (.5,1) to [out=0,in=270] (1,1.5);
\draw (-.5,0) rectangle (2.5,1.5);
\draw [very thick] (1.5,0) to [out=90,in=270] (2,1.5);
\end{tikzpicture}
\quad, \quad
\begin{tikzpicture} [scale=.4,anchorbase]
\draw (-1.5,0) rectangle (3.5,1.5);
\draw [very thick] (0,0) to [out=90,in=180] (.5,.5) to [out=0,in=90] (1,0);
\draw [very thick] (0,1.5) to [out=270,in=180] (.5,1) to [out=0,in=270] (1,1.5);
\draw [very thick] (-1,0) to [out=90,in=270] (2,1.5);
\draw [very thick] (2,0) to [out=90,in=180] (2.5,.5) to [out=0,in=90] (3,0);
\end{tikzpicture}
\otimes \begin{tikzpicture} [scale=.4,anchorbase]
\draw (-1.5,0) rectangle (3.5,1.5);
\draw [very thick] (-1,1.5) to [out=270,in=180] (.5,.5) to [out=0,in=270] (2,1.5);
\draw [very thick] (0,1.5) to [out=270,in=180] (.5,1) to [out=0,in=270] (1,1.5);
\draw [very thick] (1.5,0) to [out=90,in=270] (3,1.5);
\end{tikzpicture} 
=\begin{tikzpicture} [scale=.4,anchorbase]
\draw (-1.5,0) rectangle (7,1.5);
\draw [very thick] (0,0) to [out=90,in=180] (.5,.5) to [out=0,in=90] (1,0);
\draw [very thick] (0,1.5) to [out=270,in=180] (.5,1) to [out=0,in=270] (1,1.5);
\draw [very thick] (-1,0) to [out=90,in=270] (2,1.5);
\draw [very thick] (2,0) to [out=90,in=180] (2.5,.5) to [out=0,in=90] (3,0);
\draw [very thick] (2.5,1.5) to [out=270,in=180] (4,.5) to [out=0,in=270] (5.5,1.5);
\draw [very thick] (3.5,1.5) to [out=270,in=180] (4,1) to [out=0,in=270] (4.5,1.5);
\draw [very thick] (5,0) to [out=90,in=270] (6.5,1.5);
\end{tikzpicture}
\]
Stacking the two displayed matchings produces a closed component, which is reduced to $-2$.
\end{example} 

We write $\TL_n:=\TL(n,n)$ for the Temperley-Lieb algebra on $n$ strands. It has a presentation with generators $U_i$ for $1\leq i \leq n-1$, which correspond to cap-cup matchings between the $i$-th and $i+1$-st
 strand. A complete set of relations is: 
\begin{itemize}
\item $U_iU_{i+1}U_i=U_i$ for $1\leq i\leq n-2$ and $U_iU_{i-1}U_i=U_i$ for $2\leq i\leq n-1$,
 \item $[U_i,U_j]=0$ for $|i-j|>1$ and 
\item $U_i^2 = -2U_i$.
\end{itemize} 

Gluing such boxes along their left and right boundaries, we get Temperley-Lieb diagrams in an annulus $A=\Ss^1\times [0,1]$ together with additional information: we remember the corners of the box as base points on the boundary circles and the glued edge as an arc connecting them.
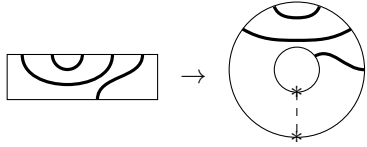
\[
\begin{tikzpicture} [scale=.4,anchorbase]
\draw (-1.5,0) rectangle (3.5,1.5);
\draw [very thick] (-1,1.5) to [out=270,in=180] (.5,.5) to [out=0,in=270] (2,1.5);
\draw [very thick] (0,1.5) to [out=270,in=180] (.5,1) to [out=0,in=270] (1,1.5);
\draw [very thick] (1.5,0) to [out=90,in=270] (3,1.5);
\end{tikzpicture} 
\;\;\rightarrow\;\;  
\begin{tikzpicture}[anchorbase, scale=.3]
\draw (0,0) circle (1);
\draw (0,0) circle (3);
\draw [very thick] (2.4,1.8) to [out=210,in=0] (0,1.3) to [out=180,in=330] (-2.4,1.8);
\draw [very thick] (1,2.8) to [out=250,in=0] (0,2.3) to [out=180,in=290] (-1,2.8);
\draw [very thick] (.8,.6) to [out=30,in=180](3,0);
\node at (0,-1) {$*$};
\node at (0,-3) {$*$};
\draw [dashed] (0,-1) to (0,-3);
\end{tikzpicture}
\]

\begin{definition} \label{def:ATL}
The affine Temperley-Lieb category, denoted $\ATL$, is the $\C$-linear category with objects $n\in \N$ and with morphism spaces $\ATL(m,n)$ spanned by crossingless tangles in the annulus $A$ with $m$ and $n$ endpoints on the inside and outside boundary respectively, without inessential (i.e. homologically trivial) circles. The multiplication is defined on basis elements by stacking an annular diagram inside another one such that endpoints and basepoints match up, and then reducing inessential circles to $-2$. We write $\ATL_n:=\ATL(n,n)$ for the affine Temperley-Lieb algebra on $n$ strands.
\end{definition}

Note that the affine Temperley-Lieb category can be obtained from the usual one by simply adding a new type of invertible morphism:
\[
\wrap 
\;\;=\;\;
\begin{tikzpicture}[anchorbase, scale=.3]
\draw (0,0) circle (1);
\draw (0,0) circle (3);
\draw [very thick] (-.8,.6) .. controls (-2.6,1.2) and (-2,-2) .. (0,-2) .. controls (2,-2) and (1.2,.9) ..  (2.4,1.8);
\draw [very thick] (-.4,.9) .. controls (-.8,1.8) and (-.9,1.2) .. (-1.8,2.4);
\node at (0,2) {$\dots$};
\draw [very thick] (.8,.6) .. controls (1.2,.9) and (.75,1.3) .. (1.5,2.6);
\node at (0,-1) {$*$};
\node at (0,-3) {$*$};
\draw [dashed] (0,-1) to (0,-3);
\end{tikzpicture}
\quad,\quad
\wrapi
\;\;=\;\;
\begin{tikzpicture}[anchorbase, scale=.3]
\draw (0,0) circle (1);
\draw (0,0) circle (3);
\draw [very thick] (.8,.6) .. controls (2.6,1.2) and (2,-2) .. (0,-2) .. controls (-2,-2) and (-1.2,.9) ..  (-2.4,1.8);
\draw [very thick] (.4,.9) .. controls (.8,1.8) and (.9,1.2) .. (1.8,2.4);
\node at (0,2) {$\dots$};
\draw [very thick] (-.8,.6) .. controls (-1.2,.9) and (-.75,1.3) .. (-1.5,2.6);
\node at (0,-1) {$*$};
\node at (0,-3) {$*$};
\draw [dashed] (0,-1) to (0,-3);
\end{tikzpicture}
\]
We write $U_0=\wrap U_{1} \wrapi = \wrapi U_{n-1} \wrap$ for the cap-cup between the first and last strand. A complete set of relations for $\ATL_n$ in terms of the generators $U_i$ and the invertible generator $D$ is given by:
\begin{itemize}
\item $U_iU_{i+1}U_i=U_i= U_iU_{i-1}U_i$,
\item $[U_i,U_j]=0$ for $i\neq j\pm 1$ ,
\item $U_i^2 = -2U_i$,
\item $U_i D = D U_{i+1}$
\end{itemize} 
where all indices are taken modulo $n$.

\begin{remark} \label{rem:TL} There exists a vast literature about the Temperley-Lieb category, Temperley-Lieb algebras and their affine versions. Our version of $\ATL_n$ agrees with the one defined by Graham and Lehrer \cite{GL} at $q=1$. Jones and Reznikoff \cite{JR} also study this version, although they use shaded diagrams which encode a parity. Our $\ATL_n$ should not be confused with the Temperley-Lieb quotients of the Iwahori-Hecke algebras of affine type A, which are studied e.g. in \cite{FGr,EG, Gre}, and which appear as the proper subalgebras of our $\ATL_n$ generated by the $U_i$, see also Section~\ref{sec:decat}.  Also note that $\ATL$ encodes crossingless tangles in the annulus up to isotopies which are the identity on the boundary. Algebras of Temperley-Lieb diagrams where isotopies need not be the identity on the boundary have e.g. been studied in \cite{Jones2}.
\end{remark}

\subsection{Link with $\slnn{2}$-representation theory}

It is well-known that the Temperley-Lieb category $\TL$ describes a category of $\slnn{2}$-representations and their intertwiners. This can be traced back to \cite{RTW}. Indeed, there exists a fully faithful $\C$-linear monoidal functor $\phi$ from $\TL$ to $\rep$, which sends $n$ to the $n$-th tensor power of the vector representation $V=\C\langle v_+,v_-\rangle$. In the following we describe a version of the functor $\phi$ using the representation theory of the associative algebra $U(\slnn{2})$.

Recall that $U(\slnn{2})$ is the $\C$-algebra generated by $E,F,H$ subject to the following relations:
\begin{align}
HE-EH=2 E,\quad HF-FH=-2F,\quad  
EF-FE=H.
\end{align}
It is a Hopf algebra with coproduct and antipode determined by $\Delta(x)=1\otimes x +x\otimes 1$ and $S(x)=-x$ for $x\in \slnn{2}$.

Denote $V=\C\langle v_-,v_+\rangle$ the vector representation, with
\begin{gather*}
Ev_{+}=0, \quad Fv_+=v_-,\quad H v_+=v_+,\\
Ev_{-}=v_+, \quad Fv_-=0,\quad H v_-=-v_-,
\end{gather*}
The dual $V^*=\C\langle v_-^*, v_+^*\rangle$ is isomorphic to $V$ via the map that send $v_-^*\mapsto v_+$ and $v_+^*\mapsto -v_-$.

The $\C$-linear monoidal functor $\phi$ is determined by its images on cap and cup morphisms in $\TL$. It associates to those the negatives of the natural evaluation and co-evaluation maps for duals, composed with the isomorphism $V \cong V^*$:
\begin{align*}
\begin{tikzpicture}[anchorbase, scale=.5]
\draw [very thick] (1,2) to (1,1.7)to [out=270,in=0] (.5,1.2) to [out=180,in=270] (0,1.7) to (0,2);
\end{tikzpicture}
\quad \xrightarrow{\phi} \quad
\begin{cases}
\C \to  V \otimes V \\
1\mapsto   v_{+}\otimes v_- - v_{-}\otimes v_+
\end{cases}
\quad ,\quad
\begin{tikzpicture}[anchorbase, scale=.5]
\draw [very thick] (1,0) to (1,0.3)to [out=90,in=0] (.5,.8) to [out=180,in=90] (0,.3) to (0,0);
\end{tikzpicture}
\quad \xrightarrow{\phi} \quad
\begin{cases}
V\otimes V \to \C \\
v_{\mp}\otimes v_{\pm} \mapsto \pm 1 \\
v_\pm \otimes v_\pm \mapsto 0 
\end{cases}
\end{align*}

It is an easy exercise to check that these intertwiners of $U(\slnn{2})$-representations satisfy the relations of the Temperley-Lieb category and so $\phi$ is a well-defined monoidal functor. The exchange of two tensor factors $s\colon V\otimes V, v_{\epsilon_1}\otimes v_{\epsilon_2} \mapsto v_{\epsilon_2}\otimes v_{\epsilon_1}$ is the image under $\phi$ of the following morphism in $\TL$: 

\begin{equation}\label{eq:braid}
\begin{tikzpicture}[anchorbase, scale=.5]
\draw [very thick] (1,0) to [out=90,in=270] (0,1.5);
\draw [very thick] (0,0) to [out=90,in=270] (1,1.5);
\end{tikzpicture}
\quad := \quad
\begin{tikzpicture}[anchorbase, scale=.5]
\draw [very thick] (0,0) -- (0,1.5);
\draw [very thick] (1,0) -- (1,1.5);
\end{tikzpicture}
\;
+ \;
\begin{tikzpicture}[anchorbase, scale=.5]
\draw [very thick] (0,0) to [out=90,in=180] (.5,.5) to [out=0,in=90] (1,0);
\draw [very thick] (0,1.5) to [out=270,in=180] (.5,1) to [out=0,in=270] (1,1.5);
\end{tikzpicture}
\end{equation}

This suggestive notation encodes useful relations in $\TL$ and is compatible with planar isotopies:
\[  \begin{tikzpicture}[anchorbase,scale=.7]
    \draw [very thick] (1,0) to [out=90,in=0] (.5,.5) to [out=180,in=90] (0,0);
    \draw [very thick] (.5,0) to [out=45,in=315] (.5,1);
  \end{tikzpicture}  
  \quad=\quad 
  \begin{tikzpicture}[anchorbase,scale=.7]
    \draw [very thick] (1,0) to [out=90,in=0] (.5,.5) to [out=180,in=90] (0,0);
    \draw [very thick] (.5,0) to [out=135,in=225] (.5,1);
  \end{tikzpicture}
  \quad, \quad
  \begin{tikzpicture}[anchorbase,scale=.7]
    \draw [very thick] (1,1) to [out=270,in=0] (.5,.5) to [out=180,in=270] (0,1);
    \draw [very thick] (.5,0) to [out=45,in=315] (.5,1);
  \end{tikzpicture}  
  \quad=\quad 
  \begin{tikzpicture}[anchorbase,scale=.7]
    \draw [very thick] (1,1) to [out=270,in=0] (.5,.5) to [out=180,in=270] (0,1);
    \draw [very thick] (.5,0) to [out=135,in=225] (.5,1);
  \end{tikzpicture}   \] 
  
\begin{lemma} \label{lem:Reid} The following analogs of Reidemeister moves hold in $\TL$:
  \[
  \begin{tikzpicture}[anchorbase,scale=.7]
    \draw [very thick] (0,0) to [out=90,in=180] (.75,1.25) to [out=0,in=90] (1.25,.75) to [out=-90,in=0] (.75,.25) to [out=180,in=-90] (0,1.5); 
  \end{tikzpicture}
  \quad=\;\;-\;
  \begin{tikzpicture}[anchorbase,scale=.7]
    \draw [very thick] (0,0) -- (0,1.5); 
  \end{tikzpicture}
  \quad,\quad
  \begin{tikzpicture}[anchorbase,scale=.5]
    \draw [very thick] (0,0) to [out=90,in=-90] (1,1) to [out=90,in=-90] (0,2);
    \draw [very thick] (1,0) to [out=90,in=-90] (0,1) to [out=90,in=-90] (1,2);
  \end{tikzpicture}
\quad=\quad
  \begin{tikzpicture}[anchorbase,scale=.5]
    \draw [very thick] (0,0) -- (0,2);
    \draw [very thick] (1,0) -- (1,2);
  \end{tikzpicture}
  \quad,\quad
  \begin{tikzpicture}[anchorbase,scale=.5]
    \draw [very thick] (0,0) to [out=90,in=-90] (.7,.7) to [out=90,in=-90] (1.4,1.4) -- (1.4,2.1);
    \draw [very thick] (.7,0) to [out=90,in=-90] (0,.7) -- (0,1.4) to [out=90,in=-90] (.7,2.1);
    \draw [very thick] (1.4,0) -- (1.4,.7) to [out=90,in=-90] (.7,1.4) to [out=90,in=-90] (0,2.1);
  \end{tikzpicture}
\quad=\quad
  \begin{tikzpicture}[anchorbase,scale=.5]
    \draw [very thick] (1.4,0) to [out=90,in=-90] (.7,.7) to [out=90,in=-90] (0,1.4) -- (0,2.1);
    \draw [very thick] (.7,0) to [out=90,in=-90] (1.4,.7) -- (1.4,1.4) to [out=90,in=-90] (.7,2.1);
    \draw [very thick] (0,0) -- (0,.7) to [out=90,in=-90] (.7,1.4) to [out=90,in=-90] (1.4,2.1);
  \end{tikzpicture}
  \]
\end{lemma}
This implies that $\TL$ is a symmetric monoidal category and that $\phi$ is a symmetric monoidal functor. 

\begin{remark} 
The tensor product $V^{\otimes m}$ decomposes into a direct sum of irreducible representations, amongst which $\sym^m(V)$ appears with multiplicity $1$ and has highest weight. This decomposition can be translated in terms of idempotents in $\TL_n$ and the idempotent corresponding to the copy of $\sym^m(V)$ is the famous Jones-Wenzl projector $P_m$ \cite{Jones,Wen} with diagrammatic representation recursively defined by $P_1=\id_1$ and:
\begin{equation}  \label{eq:JWrec}
 \begin{tikzpicture}[anchorbase, scale=.3]
\fill[black,opacity=.2] (0,1) rectangle (3,3);
\draw[thick] (0,1) rectangle (3,3);
\draw [very thick] (.5,0) to (.5,1);
\draw [thick, dotted] (.7,0.5) to (1.3,.5);
\draw [very thick] (1.5,0) to (1.5,1);
\draw [very thick] (2.5,0) to (2.5,1);
\draw [very thick] (.5,3) to (.5,4);
\draw [thick, dotted] (.7,3.5) to (1.3,3.5);
\draw [very thick] (1.5,3) to (1.5,4);
\draw [very thick] (2.5,3) to (2.5,4);
\node at (1.5,1.9) {$P_{m+1}$};
\end{tikzpicture}
\;\;=\;\;
 \begin{tikzpicture}[anchorbase, scale=.3]
\fill[black,opacity=.2] (0,1) rectangle (2,3);
\draw[thick] (0,1) rectangle (2,3);
\draw [very thick] (.5,0) to (.5,1);
\draw [thick, dotted] (.7,0.5) to (1.3,.5);
\draw [very thick] (1.5,0) to (1.5,1);
\draw [very thick] (2.5,0) to (2.5,4);
\draw [very thick] (.5,3) to (.5,4);
\draw [thick, dotted] (.7,3.5) to (1.3,3.5);
\draw [very thick] (1.5,3) to (1.5,4);
\node at (1,1.9) {$P_{m}$};
\end{tikzpicture}
\;\;+\;\frac{m}{m+1}\;
 \begin{tikzpicture}[anchorbase, scale=.3]
\fill[black,opacity=.2] (0,.5) rectangle (2,1.5);
\draw[thick] (0,.5) rectangle (2,1.5);
\fill[black,opacity=.2] (0,2.5) rectangle (2,3.5);
\draw[thick] (0,2.5) rectangle (2,3.5);
\draw [very thick] (.5,0) to (.5,.5);
\draw [thick, dotted] (.7,0.25) to (1.3,.25);
\draw [very thick] (1.5,0) to (1.5,.5);
\draw [very thick] (2.5,0) to (2.5,1.5)to [out=90,in=90] (1.5,1.5);
\draw [very thick] (1.5,2.5) to [out=270,in=270] (2.5,2.5) to (2.5,4); 
\draw [very thick] (.5,1.5) to (.5,2.5);
\draw [thick, dotted] (.7,2) to (1.3,2);
\draw [very thick] (.5,3.5) to (.5,4);
\draw [thick, dotted] (.7,3.75) to (1.3,3.75);
\draw [very thick] (1.5,3.5) to (1.5,4);
\node at (1,.95) {\tiny$P_{m}$};
\node at (1,2.95) {\tiny$P_{m}$};
\end{tikzpicture}
\end{equation}
\end{remark}

Dual to the decomposition into simple representations, one can decompose $V^{\otimes m}$ as the direct sum of its weight spaces. Of course, the projections onto weight spaces are usually not $U(\slnn{2})$-equivariant, and then cannot be realized as the images of idempotents in the Temperley-Lieb category $\TL$.

We will thus extend the functor $\phi$ from $\TL$ to $\ATL$ by sending the morphisms $D\in \ATL_m$ to certain vector space endomorphisms of $V^{\otimes m}$ which respect the weight space decomposition, but which break the $U(\slnn{2})$-action\footnote{Setting $\phi(D)$ to be a certain $U(\slnn{2})$-intertwiner can be used to construct evaluation representations of certain affine Lie algebras. We stress that this is not what we have in mind.}. Since we want $\ATL$ to inherit the monoidal structure of $\TL$, the $\phi$-image of $D\in \ATL_1$ actually determines $\phi(D)$ for all $D\in \ATL_m$. The requirement that $\phi(D)$ is invertible and respects the weight space decomposition of $V$ implies that $\phi(D)(v_\pm)= c_\pm v_\pm$ for some $c_\pm\in \C^*$. The proof of Lemma~\ref{lem:welldef} shows that we need $c_+=c_-^{-1}$ in order to obtain a well-defined functor that respects all isotopies of diagrams. The reason for the following choice will become clear in Proposition~\ref{prop:nff}.
  
\begin{definition}
  Let $V\otimes W$ be the image under $\phi$ of the domain of $D$ and $W\otimes V$ its co-domain. Then we define $\phi(D)$ to be the linear map determined by $v_{\pm}\otimes w\mapsto i^{\pm 1} w\otimes v_{\pm}
  $
  for $v_{\pm}\in V$ and any $w\in W$. Furthermore we set $\phi(\wrapi)=\phi(\wrap)^{-1}$.
\end{definition}

Let $\h$ denote the Cartan subalgebra of diagonal matrices in $\slnn{2}$ and consider $U(\h)=\langle H \rangle \subset U(\slnn{2})$. We denote by $\reph$ the category of finite-dimensional $U(\h)$-representations of integral weights, i.e. vector spaces with a $\Z$-grading. Note that the inclusion $\h\hookrightarrow \slnn{2}$ induces a restriction functor $\rep\to\reph$ and that $\phi(\wrap)$ and $\phi(\wrapi)$ are morphisms in $\reph$.

\begin{lemma}\label{lem:welldef} The functor $\phi\colon \ATL \to \reph$ is well-defined. 
\end{lemma}
\begin{proof} All morphisms in $\ATL$ are compositions of caps or cups between adjacent strands, as well as the morphisms $\wrap$ and $\wrapi$. Any relation satisfied by compositions of these generating morphisms is either supported in $\TL\subset \ATL$ (and is thus respected by $\phi$) or involves some generators $\wrap$ and $\wrapi$. Since $\phi$ maps $\wrap$ and $\wrapi$ to inverse isomorphisms, it suffices to check that $\phi$ respects the relation that $\wrap$ intertwines adjacent caps and cups. The only interesting case corresponds to isotoping caps or cups across the dashed arc. For cups we need the relation:

\[\phi
\left(
\begin{tikzpicture}[anchorbase, scale=.3]
\draw (0,0) circle (1);
\draw (0,0) circle (3);
\draw [very thick] (-.8,.6) to [out=150,in=90] (-1.75,0) to[out=270,in=180] (0,-2) to [out=0,in=270] (2,0) to [out=90,in
=225] (2.25,.9) to [out=45,in=315] (2.25,1.35)to[out=135,in=45] (1.8,1.35) to (.8,.6) ;
\draw [very thick] (-.4,.9)to  (-1.2,2.7);
\node at (0,2) {$\dots$};
\draw [very thick] (.4,.9)  to (1.2,2.7);
\node at (0,-1) {$*$};
\node at (0,-3) {$*$};
\draw [dashed] (0,-1) to (0,-3);
\end{tikzpicture}
\right)
\;=\;
\phi\left(
\begin{tikzpicture}[anchorbase, scale=.3]
\draw (0,0) circle (1);
\draw (0,0) circle (3);
\draw [very thick] (.8,.6) to [out=30,in=90] (1.75,0) to[out=270,in=0] (0,-2) to [out=180,in=270] (-2,0) to [out=90,in
=315] (-2.25,.9) to [out=135,in=225] (-2.25,1.35)to[out=45,in=135] (-1.8,1.35) to (-.8,.6) ;
\draw [very thick] (-.4,.9)to  (-1.2,2.7);
\node at (0,2) {$\dots$};
\draw [very thick] (.4,.9)  to (1.2,2.7);
\node at (0,-1) {$*$};
\node at (0,-3) {$*$};
\draw [dashed] (0,-1) to (0,-3);
\end{tikzpicture}
\right)
\]
Note that both sides annihilate vectors of the form $v_{\pm}\otimes w \otimes v_{\pm}$. On the other hand, the action on $v_{\pm}\otimes w \otimes v_{\mp}$ of the left-hand side and right-hand side, respectively, are given by: 
\begin{align*}
v_{\pm}\otimes w \otimes v_{\mp} &\mapsto  i^{\pm 1} w \otimes v_{\mp} \otimes v_{\pm} \mapsto  i w \quad \text{and} \quad
v_{\pm}\otimes w \otimes v_{\mp} \mapsto  i^{\pm 1} v_{\mp}\otimes v_{\pm}\otimes w \mapsto  i w 
\end{align*}
The isotopy relation for cups is checked analogously. 
\end{proof}

Let $\otimes\colon \ATL\times \ATL\to \ATL$ denote the bi-functor given on objects by $(m,n)\mapsto m+n$ and on morphisms by superposing a pair of affine Temperley-Lieb diagrams $(W_1,W_2)$:

\[
 \begin{tikzpicture}[anchorbase, scale=.25]
\draw[thick] (0,0) circle (2.5);
\fill[black,opacity=.2] (0,0) circle (2.5);
\draw[thick,fill=white] (0,0) circle (1.5);
\draw[thick] (0,0) circle (4.5);
\fill[black,opacity=.2] (4.5,0) arc (0:360:4.5) -- (3.5,0) arc (360:0:3.5);
\draw [thick] (0,0) circle (3.5);
\draw (0,0) circle (1);
\draw (0,0) circle (5);
\draw[dotted] (-2.29,2.29) to [out=225,in=90] (-3.25,0) to [out=270,in=135] (-2.29,-2.29);
\draw[dotted] (-3.4,3.4) to [out=225,in=90] (-4.75,0) to [out=270,in=135] (-3.4,-3.4);
\draw[dotted] (-1.93,1.93) to [out=225,in=90] (-2.75,0) to [out=270,in=135] (-1.93,-1.93);
\draw[dotted] (-0.88,0.88) to [out=225,in=90] (-1.25,0) to [out=270,in=135] (-0.88,-0.88);
\draw [white,line width=.15cm] (-.6,.8) to (-1.8,2.4);
\draw [white,line width=.15cm] (-.6,-.8) to (-1.8,-2.4);
\draw [very thick] (-.6,.8) to (-2.1,2.8);
\draw [very thick] (-.6,-.8) to (-2.1,-2.8);
\draw [very thick] (-2.7,3.6) to (-3,4);
\draw [very thick] (-2.7,-3.6) to (-3,-4);
\draw[dotted] (2.29,2.29) to [out=315,in=90] (3.25,0) to [out=270,in=45] (2.29,-2.29);
\draw[dotted] (3.4,3.4) to [out=315,in=90] (4.75,0) to [out=270,in=45] (3.4,-3.4);
\draw[dotted] (1.93,1.93) to [out=315,in=90] (2.75,0) to [out=270,in=45] (1.93,-1.93);
\draw[dotted] (0.88,0.88) to [out=315,in=90] (1.25,0) to [out=270,in=45] (0.88,-0.88);
\draw [very thick] (1.5,2) to (1.98,2.64);
\draw [very thick] (2.79,3.72) to (3,4);
\draw [very thick] (1.5,-2) to (1.98,-2.64);
\draw [very thick] (2.79,-3.72) to (3,-4);
\draw [very thick] (.6,-.8) to (.9,-1.2);
\draw [very thick] (.6,.8) to (.9,1.2);
\node at (0,-1) {$*$};
\node at (0,-5) {$*$};
\draw [dashed] (0,-1) to (0,-5);
\node at (0,1.95) {\tiny$W_2$};
\node at (0,3.95) {\tiny$W_1$};
\end{tikzpicture}
\]
and resolving all crossings via \eqref{eq:braid}. This is well-defined thanks to Lemma \ref{lem:Reid} and it induces a strict symmetric monoidal structure on $\ATL$ such that both the inclusion $\TL\to \ATL$ and  $\phi \colon \ATL \to \reph$ become symmetric monoidal functors. Let $\iota$ and $\iota^\prime$ denote the endo-functors of $\ATL$ given by superposing a single strand on the right or left:
\[
 \begin{tikzpicture}[anchorbase, scale=.3]
\draw[thick] (0,0) circle (2.5);
\fill[black,opacity=.2] (0,0) circle (2.5);
\draw[thick,fill=white] (0,0) circle (1.5);
\draw[dotted] (-1.93,1.93) to [out=45,in=180] (0,2.75) to [out=0,in=135] (1.93,1.93);
\draw[dotted] (-0.88,0.88) to [out=45,in=180] (0,1.25) to [out=0,in=135] (0.88,0.88);
\draw [very thick] (.8,.6) to (1.2,.9);
\draw [very thick] (-.8,.6) to (-1.2,.9);
\draw [very thick](2,1.5) to (2.4,1.8);
\draw [very thick] (-2,1.5) to (-2.4,1.8);
\draw [white, line width=.12cm] (-.8,-.6) to (-2.4,-1.8);
\draw [very thick] (-.8,-.6) to (-2.4,-1.8);
\node at (0,-1) {$*$};
\node at (0,-3) {$*$};
\draw [dashed] (0,-1) to (0,-3);
\draw (0,0) circle (1);
\draw (0,0) circle (3);
\end{tikzpicture}
\quad
\xleftarrow{\iota^\prime} 
\quad
\begin{tikzpicture}[anchorbase, scale=.3]
\draw[thick] (0,0) circle (2.5);
\fill[black,opacity=.2] (0,0) circle (2.5);
\draw[thick,fill=white] (0,0) circle (1.5);
\draw[dotted] (-1.93,1.93) to [out=45,in=180] (0,2.75) to [out=0,in=135] (1.93,1.93);
\draw[dotted] (-0.88,0.88) to [out=45,in=180] (0,1.25) to [out=0,in=135] (0.88,0.88);
\draw [very thick] (.8,.6) to (1.2,.9);
\draw [very thick] (-.8,.6) to (-1.2,.9);
\draw [very thick] (2,1.5) to (2.4,1.8);
\draw [very thick] (-2,1.5) to (-2.4,1.8);
\node at (0,-1) {$*$};
\node at (0,-3) {$*$};
\draw [dashed] (0,-1) to (0,-3);
\draw (0,0) circle (1);
\draw (0,0) circle (3);
\end{tikzpicture}
\quad
\xrightarrow{\iota}
\quad 
 \begin{tikzpicture}[anchorbase, scale=.3]
\draw[thick] (0,0) circle (2.5);
\fill[black,opacity=.2] (0,0) circle (2.5);
\draw[thick,fill=white] (0,0) circle (1.5);
\draw[dotted] (-1.93,1.93) to [out=45,in=180] (0,2.75) to [out=0,in=135] (1.93,1.93);
\draw[dotted] (-0.88,0.88) to [out=45,in=180] (0,1.25) to [out=0,in=135] (0.88,0.88);
\draw [very thick] (.8,.6) to (1.2,.9);
\draw [very thick] (-.8,.6) to (-1.2,.9);
\draw [very thick](2,1.5) to (2.4,1.8);
\draw [very thick] (-2,1.5) to (-2.4,1.8);
\draw [white, line width=.12cm] (.8,-.6) to (2.4,-1.8);
\draw [very thick] (.8,-.6) to (2.4,-1.8);
\node at (0,-1) {$*$};
\node at (0,-3) {$*$};
\draw [dashed] (0,-1) to (0,-3);
\draw (0,0) circle (1);
\draw (0,0) circle (3);
\end{tikzpicture}
\]
Note that these endo-functors can be written as $\iota^\prime(-) = (-)\otimes 1$ and $\iota^\prime(-)= \wrapi \iota(-)\wrap$.

\subsection{A quotient}
It is a key observation that the functor $\phi \colon \ATL \to \reph$ is not faithful.
\begin{proposition}\label{prop:nff}~\vspace{-.5cm}
 \begin{equation}
\label{eqn:essentialideal}
  \phi\left(
  \begin{tikzpicture}[anchorbase, scale=.3]
    \draw (0,0) circle (1);
    \draw (0,0) circle (3);
    \draw [very thick] (0,0) circle (2);
  \draw [very thick] (-.9,.4) to (-2.8,1.2);
    \node at (0,1.35) {$\cdots$};
    \node at (0,2.5) {$\cdots$};
    \draw [very thick] (.9,.4) to (2.8,1.2);
    \node at (0,-1) {$*$};
    \node at (0,-3) {$*$};
    \draw [dashed] (0,-1) to (0,-3);
  \end{tikzpicture}
  \right)
  =0.
\end{equation}
\end{proposition}

We define $\essATL$ as the quotient of $\ATL$ by the ideal generated by the elements shown in \eqref{eqn:essentialideal}, with any number of through-strings. 
 Note that the monoidal structure $\otimes$ and the functor $\phi$ descend to the quotient $\essATL$.

\begin{theorem} \label{thm:faithfulness}
  The functor $\phi:\essATL\mapsto \reph$ is fully faithful.
\end{theorem}

\begin{proof}
We consider the isomorphisms $f\colon \essATL(2n-m,m) \to \essATL(0,2n)$ and their inverses, which are given by the following operations on diagrams.
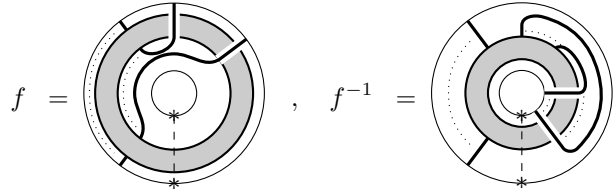
\[  f\;\;= \;\begin{tikzpicture}[anchorbase, scale=.3]
\draw[thick] (0,0) circle (3.5);
\fill[black,opacity=.2] (0,0) circle (3.5);
\draw[thick,fill=white] (0,0) circle (2.5);
\draw (0,0) circle (1);
\draw (0,0) circle (4);
\draw[dotted] (-2.65,2.65) to [out=225,in=90] (-3.75,0) to [out=270,in=135] (-2.65,-2.65);
\draw[dotted] (-1.59,1.59) to [out=225,in=90] (-2.25,0) to [out=270,in=135] (-1.59,-1.59);
\draw [very thick] (-2.4,3.2) to (-2.1,2.8);
\draw [very thick] (-2.4,-3.2) to (-2.1,-2.8);
\draw [white,line width=.15cm] (0,2.4) to (0,3.6);
\draw[very thick] (-1.5,2) to [out=300,in=270] (0,2.5) to (0,4);
\draw [white,line width=.15cm] (1.97,1.47) to (3.16,2.37);
\draw[very thick] (-1.5,-2) to [out=60,in=270] (-1.75,0) to [out=90,in=180] (0,1.75) to [out=0,in=210] (2,1.5) to (3.2,2.4);
\node at (0,-1) {$*$};
\node at (0,-4) {$*$};
\draw [dashed] (0,-1) to (0,-4);
\end{tikzpicture}  
\quad,\quad 
f^{-1}\;\;=\;
 \begin{tikzpicture}[anchorbase, scale=.3]
\draw[thick] (0,0) circle (2.5);
\fill[black,opacity=.2] (0,0) circle (2.5);
\draw[thick,fill=white] (0,0) circle (1.5);
\draw (0,0) circle (1);
\draw (0,0) circle (4);
\draw[dotted] (-2.29,2.29) to [out=225,in=90] (-3.25,0) to [out=270,in=135] (-2.29,-2.29);
\draw [very thick] (-2.4,3.2) to (-1.5,2);
\draw [very thick] (-2.4,-3.2) to (-1.5,-2);
\draw [white,line width=.15cm] (1,0) to (2.6,0);
\draw [white,line width=.15cm] (.6,-.8) to (1.8,-2.4);
\draw [very thick] (1,0) to (2.5,0) to [out=0,in=30] (1.5,2);
\draw [very thick] (.6,-.8) to (1.8,-2.4) to[out=300,in=225]  (2.46,-2.46) to [out=45,in=285] (3.38,.9) to[out=105,in=345] (.9,3.38) to [out=165,in=90] (0,3) to (0,2.5);
\draw[dotted] (1.94,1.94) to [out=135,in=0] (0,2.75) ;
\draw[dotted] (1.94,-1.94) to [out=45,in=270] (2.75,0) ;
\draw[dotted]  (1.25,0) to [out=270,in=45] (0.88,-0.88);
\node at (0,-1) {$*$};
\node at (0,-4) {$*$};
\draw [dashed] (0,-1) to (0,-4);
\end{tikzpicture}
\]
Since these operations are given by tensoring with an identity morphism and then pre-composing with cups, or post-composing with caps, there are corresponding isomorphisms $\phi(f^{\pm 1})$ between the morphism spaces $\Hom_{\reph}(V^{\otimes 2n-m},V^{\otimes m})$ and $\Hom_{\reph}(\C,V^{\otimes 2n})$.
    To prove the theorem, it suffices to check that $\phi$ induces an isomorphism between the morphism spaces $\essATL(0,2n)$ and $\Hom_{\reph}(\C,V^{\otimes 2n})$. Note that the former is spanned by annular crossingless matchings with no closed components, of which there are $\binom{2n}{n}$. To see this, we define a bijection between annular crossingless matchings and the set of labelings of their boundary points by an equal number of labels `in' and `out'. This can be done by the following inverse rules:
  \begin{itemize}
    \item given a crossingless matching, find an isotopy representative and an orientation of the arcs so that they all turn clockwise around the annulus. Assign `in' to the tail and `out' to the head.
  \item given a labeling of boundary points, connect by an arc that only wraps in the clockwise direction around the annulus any pair of adjacent boundary points with labels `in' and `out' in that order. Then, iterate the process for all pairs of `in'- and `out'-labeled points in positions $k$ and $l$ whose neighbors in position $k+1$ and $l-1$ have already been connected by arcs. 
  \end{itemize}

On the $\reph$ side, $\Hom_{\reph}(\C,V^{\otimes 2n})\simeq (V^{\otimes 2n})_0$, the zero weight space, is of dimension $\binom{2n}{n}$, and spanned by vectors $v_\epsilon:=v_{\epsilon_1}\otimes \cdots \otimes v_{\epsilon_{2n}}$ where $\epsilon=(\epsilon_1,\dots,\epsilon_{2n})\in\{+,-\}^{2n}$ contains as many pluses as minuses. We shall  argue that $\phi$ is surjective when restricted to $\essATL(0,2n)$, which also implies injectivity for dimension reasons.

  By post-composition with the $\mathfrak{S}_{2n}$-action on both sides of the map $\phi$, it is actually enough to show that one vector $v_\epsilon$ is hit. For if one is hit, since all the other ones lie in the same $\mathfrak{S}_{2n}$-orbit, they will be hit as well. We consider the case of $v_{+-+-\dots +-}$, and we will prove the result first for $n=1$:
  \[
  \phi\left(
  \begin{tikzpicture}[anchorbase, scale=.2]
    \draw (0,0) circle (1);
    \draw (0,0) circle (3);
        \draw [very thick] (1,2.8) to [out=250,in=0] (0,1.8) to [out=180,in=290] (-1,2.8);
    \node at (0,-1) {$*$};
    \node at (0,-3) {$*$};
    \draw [dashed] (0,-1) to (0,-3);
  \end{tikzpicture}
  \right)(1)
  =
  v_{+-}-v_{-+}
  \quad,\quad
  \phi\left(
  \begin{tikzpicture}[anchorbase, scale=.2]
    \draw (0,0) circle (1);
    \draw (0,0) circle (3);
    \draw [very thick] (1,2.8) to [out=250, in=90] (2,0) to [out=-90,in=0] (0,-1.8) to [out=180,in=-90] (-2,0) to [out=90,in=290] (-1,2.8);
    \node at (0,-1) {$*$};
    \node at (0,-3) {$*$};
    \draw [dashed] (0,-1) to (0,-3);
  \end{tikzpicture}
  \right)(1)
  = i v_{+-}+iv_{-+}
  \]
  Denoting $w_1$ and $w_2$ the two generators of $\essATL(0,2)$ used above, we see that $v_{+-}$ is obtained as the image of $\phi((w_1-iw_2)/2)$ applied to $1\in \C$.
  
Now the $n$-fold diagrammatic tensor power of $(w_1-iw_2)/2$ gives a morphism in $\essATL(0,2n)$ whose $\phi$ image evaluates to $v_{+-+-\dots +-}$ on $1\in \C$. 
\end{proof}

\begin{remark}
In particular, this proves that crossingless matchings form a basis for the $\Hom$-spaces $\essATL(0,2n)$ and $\essATL(2n,0)$ that can be used to deduce explicit bases for all $\Hom$-spaces. 
\end{remark}

\begin{example}\label{exa:wraps} In $\essATL_1$ we have $\wrap=-\wrapi$ and thus $\wrap^2=-1$.
\end{example}
\begin{proof}
This follows directly from resolving the crossing in the defining relation:
\begin{equation*}
0 \quad =\quad \begin{tikzpicture}[anchorbase, scale=.3]
\draw (0,0) circle (1);
\draw (0,0) circle (3);
\draw [very thick] (0,0) circle (2);
\draw [very thick] (.9,.4) to (2.8,1.2);
\node at (0,-1) {$*$};
\node at (0,-3) {$*$};
\draw [dashed] (0,-1) to (0,-3);
\end{tikzpicture}
\quad = \quad
\begin{tikzpicture}[anchorbase, scale=.3]
\draw [very thick] (.9,.4) to (1.08,.48) to [out=30,in=0] (0, 1.75) to [out=180,in=90]  (-2,0) to [out=270,in=180]  (0,-2) to [out=0,in=270] (2.25,0) to [out=90,in=210](2.8,1.2);
\node at (0,-1) {$*$};
\node at (0,-3) {$*$};
\draw [dashed] (0,-1) to (0,-3);
\draw (0,0) circle (1);
\draw (0,0) circle (3);
\end{tikzpicture}
\;\;+ \;\;
\begin{tikzpicture}[anchorbase, scale=.3]
\draw [very thick] (.9,.4) to (1.08,.48) to [out=30,in=90] (1.75,0) to [out=270,in=0]  (0,-2) to [out=180,in=270]  (-2,0) to [out=90,in=180] (0,2.25) to [out=0,in=210](2.8,1.2);
\node at (0,-1) {$*$};
\node at (0,-3) {$*$};
\draw [dashed] (0,-1) to (0,-3);
\draw (0,0) circle (1);
\draw (0,0) circle (3);
\end{tikzpicture} \vspace{-.6cm}
\end{equation*}
\end{proof}

\subsection{Extremal weight projectors}

\begin{definition}\label{def:T} The elements $T_m\in\essATL_{m}$ are recursively defined via $T_{m+1}=\iota(T_m)s_m\iota(T_m)$ for $m\geq 2$ with initial conditions $T_1=\id_1$ and $T_2=\id_2+U_1/2 + U_0/2$:
\begin{equation}
\label{eq:recursivestart}
T_2
\;\;=\;\;
 \begin{tikzpicture}[anchorbase, scale=.3]
\draw (0,0) circle (1);
\draw (0,0) circle (3);
\draw [very thick] (.8,.6) to (2.4,1.8);
\draw [very thick] (-.8,.6) to (-2.4,1.8);
\node at (0,-1) {$*$};
\node at (0,-3) {$*$};
\draw [dashed] (0,-1) to (0,-3);
\end{tikzpicture}
+\frac{1}{2} \;
 \begin{tikzpicture}[anchorbase, scale=.3]
\draw (0,0) circle (1);
\draw (0,0) circle (3);
\draw [very thick] (.8,.6) to [out=45, in=0] (0,1.5);
\draw [very thick] (-.8,.6) to [out=135,in=180] (0,1.5);
\draw [very thick] (0,2.25) to [out=0,in=225](2.4,1.8);
\draw [very thick] (0,2.25) to [out=180,in=315](-2.4,1.8);
\node at (0,-1) {$*$};
\node at (0,-3) {$*$};
\draw [dashed] (0,-1) to (0,-3);
\end{tikzpicture}
+\frac{1}{2} \;
 \begin{tikzpicture}[anchorbase, scale=.3]
\draw (0,0) circle (1);
\draw (0,0) circle (3);
\draw [very thick] (.8,.6) to [out=45, in=90] (1.5,0) to [out=270,in=0] (0,-1.5);
\draw [very thick] (-.8,.6) to [out=135,, in=90] (-1.5,0) to [out=270,in=180] (0,-1.5);
\draw [very thick] (0,-2.25) to [out=0,in=270] (2.25,0) to [out=90,in=225](2.4,1.8);
\draw [very thick] (0,-2.25) to [out=180,in=270] (-2.25,0) to [out=90,in=315](-2.4,1.8);
\node at (0,-1) {$*$};
\node at (0,-3) {$*$};
\draw [dashed] (0,-1) to (0,-3);
\end{tikzpicture}
\quad=\quad
\frac{1}{2}\left(  
\;
 \begin{tikzpicture}[anchorbase, scale=.3]
\draw (0,0) circle (1);
\draw (0,0) circle (3);
\draw [very thick] (.8,.6) to [out=45, in=315] (0,2.25) to [out=135,in=315] (-2.4,1.8);
\draw [very thick] (-.8,.6) to [out=135,in=225] (0,2.25) to [out=45,in=225] (2.4,1.8);
\node at (0,-1) {$*$};
\node at (0,-3) {$*$};
\draw [dashed] (0,-1) to (0,-3);
\end{tikzpicture}
+ \;
 \begin{tikzpicture}[anchorbase, scale=.3]
\draw (0,0) circle (1);
\draw (0,0) circle (3);
\draw [very thick] (-.8,.6) to [out=135, in=90] (-1.5,0) to [out=270,in=135] (0,-2.25) to  [out=315,in=270] (2.25,0) to [out=90,in=225](2.4,1.8);
\draw [very thick] (.8,.6) to [out=45, in=90] (1.5,0) to [out=270,in=45] (0,-2.25) to [out=225,in=270] (-2.25,0) to [out=90,in=315](-2.4,1.8);
\node at (0,-1) {$*$};
\node at (0,-3) {$*$};
\draw [dashed] (0,-1) to (0,-3);
\end{tikzpicture}
\right)
\end{equation}
\end{definition}

The recursive relation can be depicted as follows:
\begin{equation}
\label{eq:recursive}
 \begin{tikzpicture}[anchorbase, scale=.3]
\draw[thick] (0,0) circle (4);
\fill[black,opacity=.2] (0,0) circle (4);
\draw[thick,fill=white] (0,0) circle (2);
\draw[dotted] (-1.05,1.05) to [out=45,in=180] (0,1.5) to [out=0,in=135] (1.05,1.05);
\draw [thick] (0,0) circle (2);
\draw[dotted] (-3.16,3.16) to [out=45,in=180] (0,4.5) to [out=0,in=135] (3.16,3.16);
\draw (0,0) circle (1);
\draw (0,0) circle (5);
\draw [very thick] (.8,.6) to (1.6,1.2);
\draw [very thick] (1,0) to (2,0);
\draw [very thick] (4,0) to (5,0);
\draw [very thick] (-.8,.6) to (-1.6,1.2);
\draw [very thick] (3.2,2.4) to (4,3);
\draw [very thick] (-3.2,2.4) to (-4,3);
\node at (0,-1) {$*$};
\node at (0,-5) {$*$};
\draw [dashed] (0,-1) to (0,-5);
\node at (0,2.95) {\small$T_{m+1}$};
\end{tikzpicture}
\;\; :=\;\;
 \begin{tikzpicture}[anchorbase, scale=.3]
\draw[thick] (0,0) circle (2.5);
\fill[black,opacity=.2] (0,0) circle (2.5);
\draw[thick,fill=white] (0,0) circle (1.5);
\draw[dotted] (-1.93,1.93) to [out=45,in=180] (0,2.75) to [out=0,in=135] (1.93,1.93);
\draw[dotted] (-0.88,0.88) to [out=45,in=180] (0,1.25) to [out=0,in=135] (0.88,0.88);
\draw[thick] (0,0) circle (4.5);
\fill[black,opacity=.2] (4.5,0) arc (0:360:4.5) -- (3.5,0) arc (360:0:3.5);
\draw [thick] (0,0) circle (3.5);
\draw[dotted] (-3.4,3.4) to [out=45,in=180] (0,4.75) to [out=0,in=135] (3.4,3.4);
\draw[dotted] (-2.29,2.29) to [out=45,in=180] (0,3.25) to [out=0,in=135] (2.29,2.29);
\draw (0,0) circle (1);
\draw (0,0) circle (5);
\draw [very thick] (.8,.6) to (1.2,.9);
\draw [very thick] (-.8,.6) to (-1.2,.9);
\draw [very thick] (-2,1.5) to (-2.8,2.1);
\fill [white] (1.4,.2) -- (2.6,.2) -- (2.6,-.2) -- (1.4,-.2);
\fill [white] (3.4,.2) -- (4.6,.2) -- (4.6,-.2) -- (3.4,-.2);
\draw [draw =white, double=black, thick, double distance=1.25pt] (1,0) -- (2.5,0) to [out=0,in=210] (2.8,2.1);
\draw [very thick] (2,1.5) to [out=30,in=180] (3.5,0) to (5,0);
\draw [very thick] (2.8,2.1) -- (2.8,2.1);
\draw [very thick] (3.6,2.7) to (4,3);
\draw [very thick] (-3.6,2.7) to (-4,3);
\node at (0,-1) {$*$};
\node at (0,-5) {$*$};
\draw [dashed] (0,-1) to (0,-5);
\node at (0,1.95) {\tiny$T_m$};
\node at (0,3.95) {\tiny$T_m$};
\end{tikzpicture}
\;\;=\;\;
 \begin{tikzpicture}[anchorbase, scale=.3]
\draw[thick] (0,0) circle (2.5);
\fill[black,opacity=.2] (0,0) circle (2.5);
\draw[thick,fill=white] (0,0) circle (1.5);
\draw[dotted] (-1.93,1.93) to [out=45,in=180] (0,2.75) to [out=0,in=135] (1.93,1.93);
\draw[dotted] (-0.88,0.88) to [out=45,in=180] (0,1.25) to [out=0,in=135] (0.88,0.88);
\draw[thick] (0,0) circle (4.5);
\fill[black,opacity=.2] (4.5,0) arc (0:360:4.5) -- (3.5,0) arc (360:0:3.5);
\draw [thick] (0,0) circle (3.5);
\draw[dotted] (-3.4,3.4) to [out=45,in=180] (0,4.75) to [out=0,in=135] (3.4,3.4);
\draw[dotted] (-2.29,2.29) to [out=45,in=180] (0,3.25) to [out=0,in=135] (2.29,2.29);
\draw [very thick] (.8,.6) to (1.2,.9);
\draw [very thick] (-.8,.6) to (-1.2,.9);
\draw [very thick] (-2,1.5) to (-2.8,2.1);
\draw [very thick] (2,1.5) to (2.8,2.1);
\fill [white] (1.4,.2) -- (2.6,.2) -- (2.6,-.2) -- (1.4,-.2);
\fill [white] (3.4,.2) -- (4.6,.2) -- (4.6,-.2) -- (3.4,-.2);
\draw [draw =white, double=black, thick, double distance=1.25pt] (1,0) -- (5,0);
\draw [very thick] (2.8,2.1) -- (2.8,2.1);
\draw [very thick] (3.6,2.7) to (4,3);
\draw [very thick] (-3.6,2.7) to (-4,3);
\draw (0,0) circle (1);
\draw (0,0) circle (5);
\node at (0,-1) {$*$};
\node at (0,-5) {$*$};
\draw [dashed] (0,-1) to (0,-5);
\node at (0,1.95) {\tiny$T_m$};
\node at (0,3.95) {\tiny$T_m$};
\end{tikzpicture}
\;+\;
 \begin{tikzpicture}[anchorbase, scale=.3]
\draw[thick] (0,0) circle (2.5);
\fill[black,opacity=.2] (0,0) circle (2.5);
\draw[thick,fill=white] (0,0) circle (1.5);
\draw[dotted] (-1.93,1.93) to [out=45,in=180] (0,2.75) to [out=0,in=135] (1.93,1.93);
\draw[dotted] (-0.88,0.88) to [out=45,in=180] (0,1.25) to [out=0,in=135] (0.88,0.88);
\draw[thick] (0,0) circle (4.5);
\fill[black,opacity=.2] (4.5,0) arc (0:360:4.5) -- (3.5,0) arc (360:0:3.5);
\draw [thick] (0,0) circle (3.5);
\draw[dotted] (-3.4,3.4) to [out=45,in=180] (0,4.75) to [out=0,in=135] (3.4,3.4);
\draw[dotted] (-2.29,2.29) to [out=45,in=180] (0,3.25) to [out=0,in=135] (2.29,2.29);
\draw (0,0) circle (1);
\draw (0,0) circle (5);
\draw [very thick] (.8,.6) to (1.2,.9);
\draw [very thick] (-.8,.6) to (-1.2,.9);
\draw [very thick] (-2,1.5) to (-2.8,2.1);
\fill [white] (1.4,.2) -- (2.6,.2) -- (2.6,-.2) -- (1.4,-.2);
\fill [white] (3.4,.2) -- (4.6,.2) -- (4.6,-.2) -- (3.4,-.2);
\draw [draw =white, double=black, thick, double distance=1.25pt] (1,0) -- (2.5,0) to [out=0,in=30] (2,1.5);
\draw [very thick] (2.8,2.1) to [out=210,in=105] (3.09,0.83)to [out=285,in=180] (3.5,0) to (5,0);
\draw [very thick] (2.8,2.1) -- (2.8,2.1);
\draw [very thick] (3.6,2.7) to (4,3);
\draw [very thick] (-3.6,2.7) to (-4,3);
\node at (0,-1) {$*$};
\node at (0,-5) {$*$};
\draw [dashed] (0,-1) to (0,-5);
\node at (0,1.95) {\tiny$T_m$};
\node at (0,3.95) {\tiny$T_m$};
\end{tikzpicture}
\end{equation}

\begin{theorem} \label{thm:extweightproj}
  The element $\diagrep(T_m)$ is the endomorphism of $V^{\otimes m}$ projecting onto the extremal weight space $\C\la v_{+\cdots +}, v_{-\cdots -}\ra$.
\end{theorem}
\begin{proof} For $m=1$ this is tautological and for $m=2$ we use the expression $T_2=\id_2+U_1/2 + U_0/2$ to compute:
\begin{align*}
\phi(T_2)\colon v_{\pm \pm} & \mapsto v_{\pm \pm} + 0 + 0 =v_{\pm \pm} \\
 v_{\pm \mp} &\mapsto v_{\pm \mp} +(v_{\mp \pm}- v_{\pm \mp})/2 + (- v_{\mp \pm}-v_{\pm \mp})/2=0
\end{align*}
For the induction step, we see immediately from the recursion that $\phi(T_{m+1})$ annihilates $v_{\epsilon_1 \epsilon_2\cdots \epsilon_{m+1}}$ unless $\epsilon_1=\cdots=\epsilon_{m}$. In the remaining cases we have:
\begin{align*}
\phi(\iota(T_m)s_m\iota(T_m))(v_{\pm\cdots \pm \mp}) &= \phi(\iota(T_m)s_m) (v_{\pm\cdots \pm \mp}) =\phi(\iota(T_m))(v_{\pm\cdots \mp \pm})=0\\
\phi(\iota(T_m)s_m\iota(T_m))(v_{\pm\cdots \pm \pm}) &= \phi(\iota(T_m)s_m) (v_{\pm\cdots \pm \pm}) =\phi(\iota(T_m))(v_{\pm\cdots \pm \pm})=v_{\pm\cdots \pm \pm}
\end{align*} 
So $\phi(T_{m+1})$ is the extremal weight projector.
\end{proof}

Theorem \ref{thm:extweightproj}, combined with Theorem \ref{thm:faithfulness}, shows that $\essATL$ describes weight-preserving maps for tensor products of $\slnn{2}$-representations, and that the idempotents $T_m$ realize extremal weight projectors.

\begin{lemma} \label{lem:technical}
  The elements $T_m$ satisfy the following properties: 
  \begin{enumerate}
  \item $T_m^2=T_m$
  \item $T_m s_i= s_i T_m=T_m$ for $m\geq 2$.
  \item \label{item:3} $T_m U_i= U_i T_m=0$ for $m\geq 2$ 
  \item \label{item:4} $T_m \iota^{m-n}(T_n) = \iota^{m-n}(T_n) T_m = T_m$ for $1\leq n< m$.
  \item $\wrapi T_m \wrap=T_m$.
  \end{enumerate}
\end{lemma}

\begin{proof}
The proof simply consists in applying $\phi$ and checking that the result holds in $\reph$. For example, $T_m U_i=0$ holds since $U_i$ contains a cup morphism, whose image under $\phi$ maps $1\mapsto v_{+-}-v_{-+}$, which is clearly annihilated by the extremal weight projector.
\end{proof}

We leave it to the interested reader to prove these properties diagrammatically inside $\essATL$.

\subsection{Chebyshev recursion}
\label{sec:chebrec}
We now explicitly show that the projectors $T_m$ categorify the Chebyshev polynomials of the first kind in the same sense as the classical Jones-Wenzl $P_m$ projectors categorify the Chebyshev polynomials of the second kind. Note that a different approach to the categorification of the polynomial ring and to orthogonal polynomials appears in \cite{KSa}.

\begin{definition} The split Grothendieck  group of an additive category $\mathcal{C}$ is the abelian group $K_0(\mathcal{C})$ defined as the quotient of the free abelian group spanned by the isomorphism classes $[X]$ of objects $X$ of $\mathcal{C}$, modulo the ideal generated by relations of the form $[A\oplus B]=[A]+[B]$ for objects $A$, $B$ of $\mathcal{C}$. 

If $\mathcal{C}$ is monoidal, then $K_0(\mathcal{C})$ inherits a unital ring structure with multiplication $[A]\cdot[B]:=[A\otimes B]$.
\end{definition}
We have already mentioned the classical fact that $\reprs$ is isomorphic to the polynomial ring $\Z[X]$ generated by the class $X=[V]$ of the vector representation. The isomorphisms
\[\sym^m(V)\otimes V \cong \sym^{m+1}(V) \oplus \sym^{m-1}(V) \text{ for } m\geq 1\] 
imply that classes of the simple representations are given by the Chebyshev polynomials of the second kind $[\sym^m V]=J_m \in \Z[X]$. This can also be seen in the Grothendieck group of $\TL$, at the cost of passing to the Karoubi envelope. 

\begin{definition}
The Karoubi envelope of a category $\mathcal{C}$ is the category $Kar(\mathcal{C})$ with objects given by pairs $(X,e)$, where $X$ is an object of $\mathcal{C}$ and $e\in \Hom_\mathcal{C}(X,X)$ an idempotent. Morphisms between $(X,e)$ and $(Y,f)$ are of the form $f\circ g \circ e$ with $g\in \Hom_\mathcal{C}(X,Y)$. If $\mathcal{C}$ is additive or monoidal, then $Kar(\mathcal{C})$ inherits these structures.
\end{definition}
Recall that $\TL$ is equivalent to the full subcategory of $\rep$ with objects given by $V^{\otimes m}$. Since any objects of $\rep$ appears as a direct summand of some $V^{\otimes m}$, it can be picked out by an idempotent endomorphism of $V^{\otimes m}$. Thus $\Kar(\TL)$ is equivalent to $\rep$ and $K_0(\Kar(\TL))\cong \Z[X]$. Analogously, $\Kar(\essATL)$ is equivalent to $\reph$. We now explicitly compute that the classes of the Jones-Wenzl projectors satisfy the recursion relation of the Chebyshev polynomials of the second kind.
\begin{lemma} \label{lem:K0JW}
  In $K_0(Kar(\TL))$, we have $[P_m][P_1]=[P_m \otimes P_1]=[P_{m+1}]+[P_{m-1}]$. Here and in the following, we abuse notation and write $P_m$ for the object of $\Kar(\TL)$ given by the pair $(m,P_m)$. 
\end{lemma}
\begin{proof} 
  We rewrite the recurrence~\eqref{eq:JWrec} by subtracting the term with two projectors from both sides. The result is a decomposition of the idempotent $P_m\otimes P_1$ into a sum of orthogonal idempotents in $\TL$, which induces an isomorphism between objects in $\Kar(\TL)$. After applying $K_0$, we get: 
  \[
    \left[
       \begin{tikzpicture}[anchorbase, scale=.3]
\fill[black,opacity=.2] (0,1) rectangle (2,3);
\draw[thick] (0,1) rectangle (2,3);
\draw [very thick] (.5,0) to (.5,1);
\draw [thick, dotted] (.7,0.5) to (1.3,.5);
\draw [very thick] (1.5,0) to (1.5,1);
\draw [very thick] (2.5,0) to (2.5,4);
\draw [very thick] (.5,3) to (.5,4);
\draw [thick, dotted] (.7,3.5) to (1.3,3.5);
\draw [very thick] (1.5,3) to (1.5,4);
\node at (1,1.9) {$P_{m}$};
\end{tikzpicture}\,
      \right]
   \;\; = \;\;
    \left[
      \begin{tikzpicture}[anchorbase, scale=.3]
\fill[black,opacity=.2] (0,1) rectangle (3,3);
\draw[thick] (0,1) rectangle (3,3);
\draw [very thick] (.5,0) to (.5,1);
\draw [thick, dotted] (.7,0.5) to (1.3,.5);
\draw [very thick] (1.5,0) to (1.5,1);
\draw [very thick] (2.5,0) to (2.5,1);
\draw [very thick] (.5,3) to (.5,4);
\draw [thick, dotted] (.7,3.5) to (1.3,3.5);
\draw [very thick] (1.5,3) to (1.5,4);
\draw [very thick] (2.5,3) to (2.5,4);
\node at (1.5,1.9) {$P_{m+1}$};
\end{tikzpicture}
      \right]
    \;+\;
    \left[-\frac{m}{m+1}
       \begin{tikzpicture}[anchorbase, scale=.3]
\fill[black,opacity=.2] (0,.5) rectangle (2,1.5);
\draw[thick] (0,.5) rectangle (2,1.5);
\fill[black,opacity=.2] (0,2.5) rectangle (2,3.5);
\draw[thick] (0,2.5) rectangle (2,3.5);
\draw [very thick] (.5,0) to (.5,.5);
\draw [thick, dotted] (.7,0.25) to (1.3,.25);
\draw [very thick] (1.5,0) to (1.5,.5);
\draw [very thick] (2.5,0) to (2.5,1.5)to [out=90,in=90] (1.5,1.5);
\draw [very thick] (1.5,2.5) to [out=270,in=270] (2.5,2.5) to (2.5,4); 
\draw [very thick] (.5,1.5) to (.5,2.5);
\draw [thick, dotted] (.7,2) to (1.3,2);
\draw [very thick] (.5,3.5) to (.5,4);
\draw [thick, dotted] (.7,3.75) to (1.3,3.75);
\draw [very thick] (1.5,3.5) to (1.5,4);
\node at (1,.95) {\tiny$P_{m}$};
\node at (1,2.95) {\tiny$P_{m}$};
\end{tikzpicture}\,
       \right]    \]
On the left-hand side we already see $[P_m \otimes P_1]$ and the first term on the right-hand side is $[P_{m+1}]$. It now remains to prove that the idempotent shown in the second bracket on the right-hand side is isomorphic to $P_{m-1}$ in $\Kar(\TL)$. In order to avoid confusion, we return to the pair-notation in $\Kar(\TL)$. The desired isomorphism in $\Kar(\TL)$ is given by:
 \begin{equation}\label{eq:JWPiso}
      \left(
      m+1,-\frac{m}{m+1}
       \begin{tikzpicture}[anchorbase, scale=.3]
\fill[black,opacity=.2] (0,.5) rectangle (2,1.5);
\draw[thick] (0,.5) rectangle (2,1.5);
\fill[black,opacity=.2] (0,2.5) rectangle (2,3.5);
\draw[thick] (0,2.5) rectangle (2,3.5);
\draw [very thick] (.5,0) to (.5,.5);
\draw [thick, dotted] (.7,0.25) to (1.3,.25);
\draw [very thick] (1.5,0) to (1.5,.5);
\draw [very thick] (2.5,0) to (2.5,1.5)to [out=90,in=90] (1.5,1.5);
\draw [very thick] (1.5,2.5) to [out=270,in=270] (2.5,2.5) to (2.5,4); 
\draw [very thick] (.5,1.5) to (.5,2.5);
\draw [thick, dotted] (.7,2) to (1.3,2);
\draw [very thick] (.5,3.5) to (.5,4);
\draw [thick, dotted] (.7,3.75) to (1.3,3.75);
\draw [very thick] (1.5,3.5) to (1.5,4);
\node at (1,.95) {\tiny$P_{m}$};
\node at (1,2.95) {\tiny$P_{m}$};
\end{tikzpicture}
       \right)
       \;
\xy
(0,0)*{ 
\begin{tikzpicture}[anchorbase, scale=.3]
\draw[->] (-3,.2)to(3,.2);
\draw[<-] (-3,-.2)to(3,-.2);
\end{tikzpicture}
};
(0,8)*{ 
\begin{tikzpicture}[anchorbase, scale=.3]
\fill[black,opacity=.2] (0,.5) rectangle (2,1.5);
\draw[thick] (0,.5) rectangle (2,1.5);
\draw [very thick] (.5,0) to (.5,.5);
\draw [thick, dotted] (.7,0.25) to (1.3,.25);
\draw [very thick] (1.5,0) to (1.5,.5);
\draw [very thick] (2.5,0) to (2.5,1.5)to [out=90,in=90] (1.5,1.5);
\draw [very thick] (.5,1.5) to (.5,2.5);
\draw [thick, dotted] (.7,1.8) to (1.3,1.8);
\node at (1,.95) {\tiny$P_{m}$};
       \end{tikzpicture}
};
(0,-8)*{
-\frac{m}{m+1}
       \begin{tikzpicture}[anchorbase, scale=.3]
\fill[black,opacity=.2] (0,2.5) rectangle (2,3.5);
\draw[thick] (0,2.5) rectangle (2,3.5);
\draw [very thick] (1.5,2.5) to [out=270,in=270] (2.5,2.5) to (2.5,4); 
\draw [very thick] (.5,1.5) to (.5,2.5);
\draw [thick, dotted] (.7,2) to (1.3,2);
\draw [very thick] (.5,3.5) to (.5,4);
\draw [thick, dotted] (.7,3.75) to (1.3,3.75);
\draw [very thick] (1.5,3.5) to (1.5,4);
\node at (1,2.95) {\tiny$P_{m}$};
\end{tikzpicture}
}; 
\endxy      
\;
\left(m-1,
       \begin{tikzpicture}[anchorbase, scale=.3]
\fill[black,opacity=.2] (-.25,1) rectangle (2.25,3);
\draw[thick] (-.25,1) rectangle (2.25,3);
\draw [very thick] (.5,0) to (.5,1);
\draw [thick, dotted] (.7,0.5) to (1.3,.5);
\draw [very thick] (1.5,0) to (1.5,1);
\draw [very thick] (.5,3) to (.5,4);
\draw [thick, dotted] (.7,3.5) to (1.3,3.5);
\draw [very thick] (1.5,3) to (1.5,4);
\node at (1,1.9) {\tiny$P_{m-1}$};
\end{tikzpicture}
       \right)        
\end{equation}
The verification that these maps give an isomorphism in $\Kar(TL)$ uses the fact that $P_m\circ (P_{m-1}\otimes P_1)=P_m = (P_{m-1}\otimes P_1)\circ P_m$ in $\TL$ as well as the partial trace formula:
    \[
 \begin{tikzpicture}[anchorbase, scale=.3]
\fill[black,opacity=.2] (0,1) rectangle (3,3);
\draw[thick] (0,1) rectangle (3,3);
\draw [very thick] (.5,0) to (.5,1);
\draw [thick, dotted] (.7,0.5) to (1.3,.5);
\draw [very thick] (1.5,0) to (1.5,1);
\draw [very thick] (.5,3) to (.5,4);
\draw [thick, dotted] (.7,3.5) to (1.3,3.5);
\draw [very thick] (1.5,3) to (1.5,4);
\draw [very thick] (2.5,3) to (2.5,3.5) to [out=90,in=90] (3.5,3.5) to (3.5,.5) to [out=270,in=270] (2.5,.5) to (2.5,1);
\node at (1.5,1.9) {$P_{m}$};
\end{tikzpicture}
\;=\;-\frac{m+1}{m} \;
 \begin{tikzpicture}[anchorbase, scale=.3]
\fill[black,opacity=.2] (-.25,1) rectangle (2.25,3);
\draw[thick] (-.25,1) rectangle (2.25,3);
\draw [very thick] (.5,0) to (.5,1);
\draw [thick, dotted] (.7,0.5) to (1.3,.5);
\draw [very thick] (1.5,0) to (1.5,1);
\draw [very thick] (.5,3) to (.5,4);
\draw [thick, dotted] (.7,3.5) to (1.3,3.5);
\draw [very thick] (1.5,3) to (1.5,4);
\node at (1,1.9) {\tiny$P_{m-1}$};
\end{tikzpicture}
\]
which can be proved by induction on $m$ using the defining recursion of $P_m$.
\end{proof}

In order to prove that the extremal weight projectors categorify the Chebyshev polynomials $L_m$, we need an analogous partial trace formula for the $T_m$. We define the partial trace $\pTr\colon \ATL_{n+1}\to \ATL_{n}$ on diagrams $W$ by:
\[ 
\begin{tikzpicture}[anchorbase, scale=.4]
\draw[thick] (0,0) circle (2.5);
\fill[black,opacity=.2] (0,0) circle (2.5);
\draw[thick,fill=white] (0,0) circle (1.5);
\draw[dotted] (-1.93,1.93) to [out=45,in=180] (0,2.75) to [out=0,in=135] (1.93,1.93);
\draw[dotted] (-0.88,0.88) to [out=45,in=180] (0,1.25) to [out=0,in=135] (0.88,0.88);
\draw [very thick] (.8,.6) to (1.2,.9);
\draw [very thick] (-.8,.6) to (-1.2,.9);
\draw [very thick] (2,1.5) to (2.4,1.8);
\draw [very thick] (-2,1.5) to (-2.4,1.8);
\draw [white, line width=.15cm] (1.2,-.3) to [out=270,in=180] (1.5,-.6) to (2.5,-.6) to [out=0,in=270] (2.8,-.3) ;
\draw [very thick] (1.5,0) to [out=180,in=90] (1.2,-.3) to [out=270,in=180] (1.5,-.6) to (2.5,-.6) to [out=0,in=270] (2.8,-.3) to [out=90,in=0] (2.5,0);
\node at (0,-1) {$*$};
\node at (0,-3) {$*$};
\draw [dashed] (0,-1) to (0,-3);
\draw (0,0) circle (1);
\draw (0,0) circle (3);
\node at (0,1.95) {$W$};
\end{tikzpicture}
\]

\begin{lemma} If we set $T_0=2\,\id_0$, then for $m\geq 1$ we have:
 \begin{equation} \label{eq:pTr}
  \pTr(T_m)=- T_{m-1}.
\end{equation}
\end{lemma}
\begin{proof} This is immediate for $m=1$ and $m=2$. For $m\geq 3$ it follows from the recursion \eqref{eq:recursive} and the Reidemeister I relation from Lemma~\ref{lem:Reid}.
\end{proof}

\begin{lemma}
  In $K_0(Kar(\essATL)$, we have the following:
  \begin{align*}
    [T_{1}][T_1]&=[T_{1}\otimes T_1]=[T_2]+2[\id_0]\quad \text{and} \quad \\
    [T_{m}][T_1]&=[T_{m}\otimes T_1]= [T_{m+1}]+ [T_{m-1}] \;\; \forall m\geq 2.
    \end{align*}
\end{lemma}
\begin{proof}
 For $m\geq 2$, the proof is very analogous to the one of Lemma \ref{lem:K0JW}. The defining recursion \eqref{eq:recursive} gives an orthogonal decomposition of the idempotent $T_m\otimes T_1$ into $T_{m+1}$ and an idempotent containing a cap-cup. The latter is isomorphic to $T_{m-1}$ in $\Kar(\essATL)$ through an annular version of the isomorphism in \eqref{eq:JWPiso}, although with scalar $-1$ instead of $-\frac{m}{m+1}$. The verification that this gives an isomorphism relies on the fact that $T_m$ absorbs lower order projectors, see \eqref{item:4} of Lemma~\ref{lem:technical} and the partial trace formula \eqref{eq:pTr}.
  
Finally, for $m=1$ we recall that $\id_2=T_2 -U_1/2 - U_0/2$. This gives a decomposition of $\id_2=T_1\otimes T_1$ into three orthogonal idempotents. It is also easy to check that the idempotents $-U_0/2$ and $-U_1/2$ are both isomorphic to $\id_0$ in $\Kar(\essATL)$. This finishes the proof.
\end{proof}

\begin{remark}
The representation ring of $\glnn{N}$ is isomorphic to the tensor product of two copies of $\Z[[V], [\bVn^2 V],\dots, [\bVn^N V]]$ where $V$ is the vector representation of $\glnn{N}$. It is useful to identify each copy with $\C[x_1,x_2,\dots, x_N]^{\mathfrak{S}_N}$ by sending $[\bVn^i V]$ to the elementary symmetric polynomials $e_i(x_1,x_2, \dots, x_N)$. The classes of the symmetric powers $[\sym^m V]$ then correspond to complete symmetric polynomials, the classes of the simple representation indexed by the Young diagram $\lambda$ corresponds to the Schur polynomial $s_\lambda$ and the class of the extremal weight space of $V^{\otimes m}$ corresponds to the power sum symmetric polynomial $p_m=x_1^m + x_2^m + \cdots x_N^m$. Power sum symmetric polynomials, and thus extremal weight spaces, also play an important role in the HOMFLY-PT skein algebra of the torus and its relationship to the elliptic Hall algebra \cite{MSa}.
\end{remark}

\subsection{Product formula}\label{sec:parprod} 
The subject of this section is the categorification of Equation \eqref{eq:mult}, as stated in the following theorem.
\begin{theorem}
$T_m\otimes T_n\cong T_{m+n}\oplus T_{|m-n|}$ in $\Kar(\essATL)$.
\end{theorem}
To prove this, we first split $T_m\otimes T_n$ into two orthogonal idempotents $T_{m+n}$ and $e_{m,n}$ in Lemma~\ref{lem:splitidem}. The Propositions~\ref{prop:parprod} and \ref{prop:parprod2} then identify $e_{m,n}$ with $T_{|m-n|}$ in the two distinct cases $m\neq n$ and $m=n$.

We start by noting that crossing-connected projectors can be combined.
\begin{lemma}\label{lem:linkedproj} Let $m,n\in \N$ with $m+n\geq 3$, then $(T_m\otimes T_n)s_m(T_m\otimes T_n)=T_{m+n}$.
\end{lemma}
\begin{proof}
This follows again by application of $\phi$.
\end{proof}

\begin{lemma} \label{lem:splitidem} For $m,n\geq 1$ we have an orthogonal decomposition of idempotents $T_m\otimes T_n = T_{m+n} + e_{m,n}$
 where $e_{1,1}=-U_1/2 - \wrapi U_1 \wrap/2$ and $e_{m,n}=-(T_m\otimes T_n)U_m(T_m\otimes T_n)$ otherwise.
\end{lemma}
\begin{proof}
  For $m=n=1$ this follows from the explicit description of $T_2$. Otherwise we use Lemma~\ref{lem:linkedproj}:
\[T_{m+n}=(T_m\otimes T_n)s_m(T_m\otimes T_n)= (T_m\otimes T_n) + (T_m\otimes T_n)U_m(T_m\otimes T_n).\]
Clearly, $e_{m,n}$ and $T_{m+n}$ are orthogonal by Lemma \ref{lem:technical},  (\ref{item:3}), which implies the idempotency of $e_{m,n}$.
\end{proof}

As expected, $\phi(e_{m,n})$ projects onto $\C\la v_{+^m-^n},v_{-^n+^m}\ra$. This could be used to rephrase the last proof.

\begin{lemma}\label{lem:diffproj}
For $1\leq n,m$ and $n+m\geq 3$, the projector $e_{m,n}$ can alternatively be written as 
\[e_{m,n}=(-1)^r(T_m\otimes T_n)(T_{m-r}\otimes \Cu_r\Ca_r\otimes T_{n-r})(T_m\otimes T_n)\] where $1\leq r\leq\min(m,n)$, $r<\max(m,n)$ and $\Cu_r\Ca_r$ is a composition of $r$ nested caps followed by $r$ nested cups. 
\end{lemma}
\begin{proof}
  One simply computes the images under $\phi$ and checks they agree. 
 \end{proof}

\begin{lemma}\label{lem:kariso} For $1\leq n\leq m$ we have
$(-1)^n(\id_{m-n}\otimes \Ca_n)(T_m\otimes T_n)(\id_{m-n}\otimes \Cu_n) = T_{m-n}$ where $\Ca_n$ and $\Cu_n$ denote $n$ nested caps and caps respectively and we set $\T_0=2$.  
\end{lemma}

\begin{proof} 
We first verify $\Ca_n(\id_n\otimes T_n)=\Ca_n(T_n \otimes \id_n)$. This is trivial for $n=1$ and easily verified through a short computation for $n=2$. Indeed, $T_2=\id_2 + U_1/2 + U_0/2$ and the first two terms can be isotoped across $\Ca_2$ at no cost. For the term $U_0/2$ we compare:
\begin{gather*}\Ca_2(\id_2\otimes U_0) \;=\;
 \begin{tikzpicture}[anchorbase, scale=.3]
\draw (0,0) circle (1);
\draw (0,0) circle (3);
\draw [very thick] (.8,.6)to [out=60,in=90] (1.5,0) to [out=270,in=0] (0,-1.5) to [out=180,in=270] (-1.5,0) to [out=90,in=90]  (0,1);
\draw [very thick] (0,-2.25) to [out=0,in=270] (2.25,0) to [out=90,in=270] (1.25,1.75) to[out=90,in=0](0,2.75) to [out=180,in=90] (-1.75,1) to [out=270,in=180] (-1,0);
\draw [very thick] (0,-2.25) to [out=180,in=270] (-2.25,0) to[out=90,in=180] (0,1.75) to [out=0,in=270] (.5,2) to [out=90,in=0] (0,2.25) to [out=180,in=120] (-.8,.6) ;
\node at (0,-1) {$*$};
\node at (0,-3) {$*$};
\draw [dashed] (0,-1) to (0,-3);
\end{tikzpicture}
\;=\; - \;
 \begin{tikzpicture}[anchorbase, scale=.3]
\draw (0,0) circle (1);
\draw (0,0) circle (3);
\draw [very thick] (.8,.6)to [out=60,in=90] (1.5,0) to [out=270,in=0] (0,-1.5) to [out=180,in=270] (-1.5,0) to [out=90,in=90]  (0,1);
\draw [very thick] (0,-2.25) to [out=0,in=270] (2.25,0) to [out=90,in=0](0,2.25) to [out=180,in=90] (-1.75,1) to [out=270,in=180] (-1,0);
\draw [very thick] (0,-2.25) to [out=180,in=270] (-2.25,0) to[out=90,in=120] (-.8,.6) ;
\node at (0,-1) {$*$};
\node at (0,-3) {$*$};
\draw [dashed] (0,-1) to (0,-3);
\end{tikzpicture}
\;=\; - \;
 \begin{tikzpicture}[anchorbase, scale=.3]
\draw (0,0) circle (1);
\draw (0,0) circle (3);
\draw [very thick] (.8,.6)to [out=60,in=90] (1.5,0) to [out=270,in=0] (0,-1.5) to [out=180,in=270] (-1.5,0) to [out=90,in=90]  (0,1);
\draw [very thick] (0,-2.25) to [out=0,in=270] (2.25,0) to [out=90,in=0](0,2.25) to [out=180,in=90] (-1.75,1.25) to [out=270,in=120] (-.8,.6);
\draw [very thick] (0,-2.25) to [out=180,in=270] (-2.25,0) to[out=90,in=180] (-1,0) ;
\node at (0,-1) {$*$};
\node at (0,-3) {$*$};
\draw [dashed] (0,-1) to (0,-3);
\end{tikzpicture}
\;- \;
 \begin{tikzpicture}[anchorbase, scale=.3]
\draw (0,0) circle (1);
\draw (0,0) circle (3);
\draw [very thick] (.8,.6)to [out=60,in=90] (1.5,0) to [out=270,in=0] (0,-1.5) to [out=180,in=270] (-1.5,0) to [out=90,in=90]  (0,1);
\draw [very thick] (0,-2.25) to [out=0,in=270] (2.25,0) to [out=90,in=0](0,2.25);
\draw[very thick] (-1,0) to (-1.75,0) to [out=180,in=120] (-1.4,1.05) to (-.8,.6);
\draw [very thick] (0,-2.25) to [out=180,in=270] (-2.25,0) to[out=90,in=180] (0,2.25) ;
\node at (0,-1) {$*$};
\node at (0,-3) {$*$};
\draw [dashed] (0,-1) to (0,-3);
\end{tikzpicture}
\\
\Ca_2(U_0\otimes \id_2) \;=\; 
 \begin{tikzpicture}[anchorbase, xscale=-.3, yscale=.3]
\draw (0,0) circle (1);
\draw (0,0) circle (3);
\draw [very thick] (.8,.6)to [out=60,in=90] (1.5,0) to [out=270,in=0] (0,-1.5) to [out=180,in=270] (-1.5,0) to [out=90,in=90]  (0,1);
\draw [very thick] (0,-2.25) to [out=0,in=270] (2.25,0) to [out=90,in=270] (1.25,1.75) to[out=90,in=0](0,2.75) to [out=180,in=90] (-1.75,1) to [out=270,in=180] (-1,0);
\draw [very thick] (0,-2.25) to [out=180,in=270] (-2.25,0) to[out=90,in=180] (0,1.75) to [out=0,in=270] (.5,2) to [out=90,in=0] (0,2.25) to [out=180,in=120] (-.8,.6) ;
\node at (0,-1) {$*$};
\node at (0,-3) {$*$};
\draw [dashed] (0,-1) to (0,-3);
\end{tikzpicture}
\;=\; - \;
 \begin{tikzpicture}[anchorbase, xscale=-.3, yscale=.3]
\draw (0,0) circle (1);
\draw (0,0) circle (3);
\draw [very thick] (.8,.6)to [out=60,in=90] (1.5,0) to [out=270,in=0] (0,-1.5) to [out=180,in=270] (-1.5,0) to [out=90,in=90]  (0,1);
\draw [very thick] (0,-2.25) to [out=0,in=270] (2.25,0) to [out=90,in=0](0,2.25) to [out=180,in=90] (-1.75,1) to [out=270,in=180] (-1,0);
\draw [very thick] (0,-2.25) to [out=180,in=270] (-2.25,0) to[out=90,in=120] (-.8,.6) ;
\node at (0,-1) {$*$};
\node at (0,-3) {$*$};
\draw [dashed] (0,-1) to (0,-3);
\end{tikzpicture}
\;=\; - \;
 \begin{tikzpicture}[anchorbase, xscale=-.3, yscale=.3]
\draw (0,0) circle (1);
\draw (0,0) circle (3);
\draw [very thick] (.8,.6)to [out=60,in=90] (1.5,0) to [out=270,in=0] (0,-1.5) to [out=180,in=270] (-1.5,0) to [out=90,in=90]  (0,1);
\draw [very thick] (0,-2.25) to [out=0,in=270] (2.25,0) to [out=90,in=0](0,2.25) to [out=180,in=90] (-1.75,1.25) to [out=270,in=120] (-.8,.6);
\draw [very thick] (0,-2.25) to [out=180,in=270] (-2.25,0) to[out=90,in=180] (-1,0) ;
\node at (0,-1) {$*$};
\node at (0,-3) {$*$};
\draw [dashed] (0,-1) to (0,-3);
\end{tikzpicture}
\;- \;
 \begin{tikzpicture}[anchorbase, xscale=-.3, yscale=.3]
\draw (0,0) circle (1);
\draw (0,0) circle (3);
\draw [very thick] (.8,.6)to [out=60,in=90] (1.5,0) to [out=270,in=0] (0,-1.5) to [out=180,in=270] (-1.5,0) to [out=90,in=90]  (0,1);
\draw [very thick] (0,-2.25) to [out=0,in=270] (2.25,0) to [out=90,in=0](0,2.25);
\draw[very thick] (-1,0) to (-1.75,0) to [out=180,in=120] (-1.4,1.05) to (-.8,.6);
\draw [very thick] (0,-2.25) to [out=180,in=270] (-2.25,0) to[out=90,in=180] (0,2.25) ;
\node at (0,-1) {$*$};
\node at (0,-3) {$*$};
\draw [dashed] (0,-1) to (0,-3);
\end{tikzpicture}
\end{gather*}
The second terms on the right-hand sides of both equations are zero and the first terms are isotopic. The case of $n\geq 3$ then follows inductively from the recursive definition of the projectors $T_n$ and the fact that crossings slide around caps. 

Having established that projectors slide around caps, we also have $\Ca_n(T_n\otimes T_n)=\Ca_n(T_n^2 \otimes \id_n)=\Ca_n(T_n \otimes \id_n)$ and the result follows from the partial trace formula \eqref{eq:pTr}. 
\end{proof}

\begin{proposition} 
\label{prop:parprod}For $m,n\geq 1$ and $m\neq n$ the idempotents $e_{m,n}$ and $T_{|m-n|}$ represent isomorphic objects in $\Kar(\essATL)$.
\end{proposition}
\begin{proof} We may assume that $m>n$ as the case $m<n$ is completely analogous. We use Lemma~\ref{lem:diffproj} to write $e_{m,n}=(T_m\otimes T_n)(T_{m-n}\otimes \Cu_n)(T_{m-n}\otimes Ca_n)(T_m\otimes T_n)$. Then it is immediate from Lemma~\ref{lem:kariso} that the maps $(T_m\otimes T_n)(T_{m-n}\otimes \Cu_n)$ and $(T_{m-n}\otimes \Ca_n)(T_m\otimes T_n)$ are inverse isomorphisms between the elements of the Karoubi element represented by the idempotents $e_{m,n}$ and $T_{m-n}$. 
\end{proof}

In order to describe $e_{m,m}$, we need the fact that overlapping projectors can be combined.
\begin{lemma} \label{lem:overlap}Let $1\leq n\leq m$ and $0\leq r<n$, then the following hold: 
\begin{align*}
(T_m\otimes \id_r)(\id_{m-n+r}\otimes T_n) &= T_{m+r} = (\id_{m-n+r}\otimes T_n)(T_m\otimes \id_r)\\
(\id_r \otimes T_m)(T_n\otimes \id_{m-n+r}) &= T_{m+r} =(T_n\otimes \id_{m-n+r})(\id_r \otimes T_m)
\end{align*}
\end{lemma}
\begin{proof} Again, this follows by application of $\phi$.
\end{proof}

\begin{proposition}
\label{prop:parprod2}
The idempotent $e_{m,m}$ is isomorphic to two copies of $\id_0$ in $\Kar(\essATL)$.
\end{proposition}
\begin{proof}
For $m=1$ we have $e_{1,1}=-U_1/2 - \wrapi U_1 \wrap/2$. The two summands are orthogonal idempotents, each of which is isomorphic to $\id_0$ in $\Kar(\essATL)$. For $m>1$ we rewrite:
\begin{align*} e_{m,m} =& (-1)^{m-1} (T_m\otimes T_m)(\id_1\otimes(\Cu_{m-1}\Ca_{m-1})\otimes \id_1)(T_m\otimes T_m)
\\
=& \underbrace{(-1)^{m}(T_m\otimes T_m)\Cu_{m}}_{f_1}\circ \underbrace{\Ca_{m}(T_m\otimes T_m)/2}_{g_1} 
\\&+  \underbrace{(-1)^{m}(T_m\otimes T_m)\wrap (\Cu_1\otimes \Cu_{m-1})}_{f_2} \circ \underbrace{(\Ca_1\otimes \Ca_{m-1})\wrapi(T_m\otimes T_m)/2}_{g_2} 
\end{align*} The first equality comes from Lemma~\ref{lem:diffproj}. The second equality can be verified by expanding $T_2$ in the equality $0=(T_m\otimes T_m)\wrap(T_2\otimes \Cu_{m-1}\Ca_{m-1})\wrapi(T_m\otimes T_m)$, which follows from Lemma~\ref{lem:overlap} and Lemma \ref{lem:technical}, (\ref{item:3}). To prove the proposition, it remains to verify that $g_i f_j = \delta_{i,j} \id_0$. We give one example for orthogonality:
\begin{align*}
g_1 f_2 &= (-1)^{m}\Ca_{m}(T_m\otimes T_m)\wrap (\Cu_1\otimes \Cu_{m-1})/2 \\
&= (-1)^{m}\Ca_{m}(\id_m\otimes T_m)\wrap (\Cu_1\otimes \Cu_{m-1})/2 
= -\id_0\otimes (\Ca_1 D \Cu_1) =0 
\end{align*}
Here we have used the proof of Lemma~\ref{lem:kariso}, partial trace relations and the essential torus relation. The proof of $g_2 f_1=0$ is analogous. $g_1 f_1= \id_0$ follows from equation~\eqref{eq:pTr}. It remains to check 
\[
g_2 f_2= (-1)^m(\Ca_1\otimes \Ca_{m-1})\wrapi(T_m\otimes T_m)\wrap (\Cu_1\otimes \Cu_{m-1})/2 =  \id_0.
\]
For $m=2$ this follows by expanding the right copy of $T_2$ and seeing that all terms except the identity term die. The result is evaluated using the partial trace relation twice, which produces $\id_0$. For $m\geq 3$, we use the recursion on the right copy of $T_m$, absorb the resulting copies of $T_{m-1}$ as in the proof of Lemma~\ref{lem:kariso} and then apply the partial trace relation $m-2$ times. The result  is equal to $(\Ca_1\otimes \Ca_1)\wrapi (T_2\otimes s)\wrap (\Cu_1\otimes \Cu_1)/2$, which evaluates to $\id_0$ after expanding the crossing $s$.
\end{proof}

\subsection{Decategorification}\label{sec:decat}

We have seen that the extremal weight projectors in $\essATL$ categorify the Chebyshev polynomials of the first kind in the sense that their $K_0$-classes satisfy the appropriate recurrence relation and multiplication rule. However, the Grothendieck group of the Karoubi envelope of $\essATL$, i.e. the representation ring of $\mathfrak{h}$, is actually larger than the polynomial ring $\Z[X]\cong K_0(\Kar(\TL))\cong K_0(\rep)$. To see this, recall that the objects in $\reph$ are direct sums of integral $\slnn{2}$ weight spaces. However, in the Grothendieck group, such direct sums can be written as formal differences of $\slnn{2}$-representations only if they are orbits of the action for the Weyl group $\mathfrak{S}_2$. There are two ways to address this issue: in this section, we identify a sub-category of $\Kar(\essATL)$ that is $\mathfrak{S}_2$-equivariant, that contains the extremal weight projectors and has $\Z[X]$ as Grothendieck group. In the next section, on the other hand, we use the broken symmetry in $\essATL$ to identify the projectors onto highest and lowest weight spaces.

\begin{definition} We let $\reph^{\mathfrak{S}_2}$ denote the subcategory of $\reph$ with objects that are invariant under $\mathfrak{S}_2$ and morphisms that are $\mathfrak{S}_2$-equivariant.  
\end{definition}  

\begin{lemma} The category $\reph^{\mathfrak{S}_2}$ is semi-simple, $\phi(T_m)$ for $m\geq 1$ are simple objects and the homomorphism $\Z[X]\cong  K_0(\rep) \to K_0(\reph^{\mathfrak{S}_2})$ induced by the inclusion is an isomorphism.
\end{lemma}
\begin{proof}
The indecomposable objects in $\reph^{\mathfrak{S}_2}$ are of the form $\C\langle v_{\epsilon_1,\dots, \epsilon_{n+m}}, v_{-\epsilon_1,\dots, -\epsilon_{n+m}}\rangle$. Through the permutation action, such an object is isomorphic to an object of the form $\C\langle v_{+^m,-^n}, v_{-^m,+^n})\rangle$ with $m\geq n$ and then further to $\C\langle v_{+^{m-n}}, v_{-^{m-n}})\rangle = \phi(T_{m-n})$ if $m>n$. There are no morphisms between distinct such objects and their endomorphism algebras are $1$-dimensional over $\C$. This shows that $\reph^{\mathfrak{S}_2}$ is semi-simple. The isomorphism follows since extremal weight spaces can be expressed as formal differences of the classes of simple representations in $K_0(\rep)$.
\end{proof}

We aim to describe the subcategory $\reph^{\mathfrak{S}_2}$ of $\reph$ by a subcategory of $\essATL$.

\begin{definition}
Let $\sessATL$ denote the symmetric monoidal $\C$-linear subcategory of $\essATL$ with the same objects, but with morphisms spaces spanned by compositions of the cap-cups $U_i$ together with the wrap-around caps and cups:
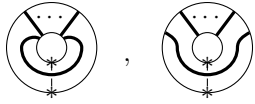
\[  \begin{tikzpicture}[anchorbase, scale=.2]
    \draw (0,0) circle (1);
    \draw (0,0) circle (3);
    \draw [very thick] (.8,.6) to [out=30, in=90] (2,0) to [out=-90,in=0] (0,-1.8) to [out=180,in=-90] (-2,0) to [out=90,in=150] (-.8,.6);
    \draw [very thick] (.6,.8) to (1.8,2.4);
    \draw [very thick] (-.6,.8) to (-1.8,2.4);
    \node at (0,-1) {$*$};
    \node at (0,2) {\small$\cdots$};
    \node at (0,-3) {$*$};
    \draw [dashed] (0,-1) to (0,-3);
  \end{tikzpicture}
\quad, \quad 
\begin{tikzpicture}[anchorbase, scale=.2]
    \draw (0,0) circle (1);
    \draw (0,0) circle (3);
    \draw [very thick] (.6,.8) to (1.8,2.4);
    \draw [very thick] (-.6,.8) to (-1.8,2.4);
    \draw [very thick] (2.8,1) to [out=210, in=90] (2,0) to [out=-90,in=0] (0,-1.8) to [out=180,in=-90] (-2,0) to [out=90,in=330] (-2.8,1);
    \node at (0,-1) {$*$};
    \node at (0,-3) {$*$};
    \node at (0,2) {\small$\cdots$};
    \draw [dashed] (0,-1) to (0,-3);
  \end{tikzpicture}\]
\end{definition}
Note that the restriction of $\phi$ to the subcategory $\sessATL$ has image contained in $\reph^{\mathfrak{S}_2}$. 

\begin{remark} The endomorphism algebra $\sessATL(n,n)$ is isomorphic to the essential-circle quotient of the Temperley-Lieb quotient of the Hecke algebra of type $\widehat{A}_{n-1}$.
\end{remark}

\begin{lemma} $T_m$, $e_{m,n}$ and the isomorphisms between $e_{m,n}$ and $T_{|m-n|}$ are contained in $\sessATL$.
\end{lemma}
\begin{proof}
This is clear for $T_m$ and $e_{m,n}$ from their definitions. The isomorphisms used in the proof of Proposition~\ref{prop:parprod} are not in $\sessATL$. However, they can be twisted by a suitable power of $\wrap$, which turns caps and cups into wrap-around caps and cups, which are in $\sessATL$. 
\end{proof}

\begin{proposition}
The functor $\phi \colon \sessATL \to \reph^{\mathfrak{S}_2}$  is fully faithful and induces an equivalence between $\Kar(\sessATL)$ and $\reph^{\mathfrak{S}_2}$. 
\end{proposition}
\begin{proof} Faithfulness is inherited from Theorem~\ref{thm:faithfulness}. Fullness follows since the image of $\phi$ contains the projections onto the simple objects in $\reph^{\mathfrak{S}_2}$. 
\end{proof}

\begin{corollary} $\Kar(\sessATL)$ categorifies the polynomial ring $\Z[X]$ and its objects $T_m$ categorify the Chebyshev polynomials of the first kind.
\end{corollary}

%
\subsection{Highest and lowest weight projectors}
%

Because of the desired topological application \cite{QW}, we have up to now focused on the idempotents that project onto the extremal weight space in a $\slnn{2}$-representation. Actually, by Theorem \ref{thm:faithfulness}, we can split $T_m$ in $\essATL$ into orthogonal idempotents that project onto the highest weight space and the lowest weight space separately. This is reminiscent of the splitting of the Chebyshev polynomials used in \cite{QRu}.

\begin{proposition}
  The following defines a family of idempotents in $\essATL$ that correspond under $\phi$ to projectors onto the highest and lowest weight spaces respectively:
  \begin{align*}
  T_{++}&=\frac{1}{4}\left(
  2\;
  \begin{tikzpicture}[anchorbase, scale=.25]
\draw (0,0) circle (1);
\draw (0,0) circle (3);
\draw [very thick] (.8,.6) to (2.4,1.8);
\draw [very thick] (-.8,.6) to (-2.4,1.8);
\node at (0,-1) {$*$};
\node at (0,-3) {$*$};
\draw [dashed] (0,-1) to (0,-3);
\end{tikzpicture}
  +2i\;
  \begin{tikzpicture}[anchorbase, scale=.25]
    \draw (0,0) circle (1);
    \draw (0,0) circle (3);
    \draw [very thick] (-.8,.6) to [out=135,in=90] (-1.5,0) to [out=-90,in=180] (0,-1.6) to [out=0,in=-90] (1.7,0) to [out=90,in=-135] (2.4,1.8);
    \draw [very thick] (.8,.6) to [out=45,in=-45] (-2.4,1.8);
    \node at (0,-1) {$*$};
    \node at (0,-3) {$*$};
    \draw [dashed] (0,-1) to (0,-3);
  \end{tikzpicture}
+\;
 \begin{tikzpicture}[anchorbase, scale=.25]
\draw (0,0) circle (1);
\draw (0,0) circle (3);
\draw [very thick] (.8,.6) to [out=45, in=0] (0,1.5);
\draw [very thick] (-.8,.6) to [out=135,in=180] (0,1.5);
\draw [very thick] (0,2.25) to [out=0,in=225](2.4,1.8);
\draw [very thick] (0,2.25) to [out=180,in=315](-2.4,1.8);
\node at (0,-1) {$*$};
\node at (0,-3) {$*$};
\draw [dashed] (0,-1) to (0,-3);
\end{tikzpicture}
+\;
 \begin{tikzpicture}[anchorbase, scale=.25]
\draw (0,0) circle (1);
\draw (0,0) circle (3);
\draw [very thick] (.8,.6) to [out=45, in=90] (1.5,0) to [out=270,in=0] (0,-1.5);
\draw [very thick] (-.8,.6) to [out=135,, in=90] (-1.5,0) to [out=270,in=180] (0,-1.5);
\draw [very thick] (0,-2.25) to [out=0,in=270] (2.25,0) to [out=90,in=225](2.4,1.8);
\draw [very thick] (0,-2.25) to [out=180,in=270] (-2.25,0) to [out=90,in=315](-2.4,1.8);
\node at (0,-1) {$*$};
\node at (0,-3) {$*$};
\draw [dashed] (0,-1) to (0,-3);
 \end{tikzpicture}
+i\;
 \begin{tikzpicture}[anchorbase, scale=.25]
\draw (0,0) circle (1);
\draw (0,0) circle (3);
\draw [very thick] (.8,.6) to [out=45, in=0] (0,1.5);
\draw [very thick] (-.8,.6) to [out=135,in=180] (0,1.5);
\draw [very thick] (0,-2.25) to [out=0,in=270] (2.25,0) to [out=90,in=225](2.4,1.8);
\draw [very thick] (0,-2.25) to [out=180,in=270] (-2.25,0) to [out=90,in=315](-2.4,1.8);
\node at (0,-1) {$*$};
\node at (0,-3) {$*$};
\draw [dashed] (0,-1) to (0,-3);
\end{tikzpicture}
+i\;
 \begin{tikzpicture}[anchorbase, scale=.25]
\draw (0,0) circle (1);
\draw (0,0) circle (3);
\draw [very thick] (.8,.6) to [out=45, in=90] (1.5,0) to [out=270,in=0] (0,-1.5);
\draw [very thick] (-.8,.6) to [out=135,, in=90] (-1.5,0) to [out=270,in=180] (0,-1.5);
\draw [very thick] (0,2.25) to [out=0,in=225](2.4,1.8);
\draw [very thick] (0,2.25) to [out=180,in=315](-2.4,1.8);
\node at (0,-1) {$*$};
\node at (0,-3) {$*$};
\draw [dashed] (0,-1) to (0,-3);
 \end{tikzpicture}
 \right)
\\
 T_{--}&=\frac{1}{4}\left(
  2\;
  \begin{tikzpicture}[anchorbase, scale=.25]
\draw (0,0) circle (1);
\draw (0,0) circle (3);
\draw [very thick] (.8,.6) to (2.4,1.8);
\draw [very thick] (-.8,.6) to (-2.4,1.8);
\node at (0,-1) {$*$};
\node at (0,-3) {$*$};
\draw [dashed] (0,-1) to (0,-3);
\end{tikzpicture}
  -2i\;
  \begin{tikzpicture}[anchorbase, scale=.25]
    \draw (0,0) circle (1);
    \draw (0,0) circle (3);
    \draw [very thick] (-.8,.6) to [out=135,in=90] (-1.5,0) to [out=-90,in=180] (0,-1.6) to [out=0,in=-90] (1.7,0) to [out=90,in=-135] (2.4,1.8);
    \draw [very thick] (.8,.6) to [out=45,in=-45] (-2.4,1.8);
    \node at (0,-1) {$*$};
    \node at (0,-3) {$*$};
    \draw [dashed] (0,-1) to (0,-3);
  \end{tikzpicture}
+\;
 \begin{tikzpicture}[anchorbase, scale=.25]
\draw (0,0) circle (1);
\draw (0,0) circle (3);
\draw [very thick] (.8,.6) to [out=45, in=0] (0,1.5);
\draw [very thick] (-.8,.6) to [out=135,in=180] (0,1.5);
\draw [very thick] (0,2.25) to [out=0,in=225](2.4,1.8);
\draw [very thick] (0,2.25) to [out=180,in=315](-2.4,1.8);
\node at (0,-1) {$*$};
\node at (0,-3) {$*$};
\draw [dashed] (0,-1) to (0,-3);
\end{tikzpicture}
+\;
 \begin{tikzpicture}[anchorbase, scale=.25]
\draw (0,0) circle (1);
\draw (0,0) circle (3);
\draw [very thick] (.8,.6) to [out=45, in=90] (1.5,0) to [out=270,in=0] (0,-1.5);
\draw [very thick] (-.8,.6) to [out=135,, in=90] (-1.5,0) to [out=270,in=180] (0,-1.5);
\draw [very thick] (0,-2.25) to [out=0,in=270] (2.25,0) to [out=90,in=225](2.4,1.8);
\draw [very thick] (0,-2.25) to [out=180,in=270] (-2.25,0) to [out=90,in=315](-2.4,1.8);
\node at (0,-1) {$*$};
\node at (0,-3) {$*$};
\draw [dashed] (0,-1) to (0,-3);
 \end{tikzpicture}
-i\;
 \begin{tikzpicture}[anchorbase, scale=.25]
\draw (0,0) circle (1);
\draw (0,0) circle (3);
\draw [very thick] (.8,.6) to [out=45, in=0] (0,1.5);
\draw [very thick] (-.8,.6) to [out=135,in=180] (0,1.5);
\draw [very thick] (0,-2.25) to [out=0,in=270] (2.25,0) to [out=90,in=225](2.4,1.8);
\draw [very thick] (0,-2.25) to [out=180,in=270] (-2.25,0) to [out=90,in=315](-2.4,1.8);
\node at (0,-1) {$*$};
\node at (0,-3) {$*$};
\draw [dashed] (0,-1) to (0,-3);
\end{tikzpicture}
-i\;
 \begin{tikzpicture}[anchorbase, scale=.25]
\draw (0,0) circle (1);
\draw (0,0) circle (3);
\draw [very thick] (.8,.6) to [out=45, in=90] (1.5,0) to [out=270,in=0] (0,-1.5);
\draw [very thick] (-.8,.6) to [out=135,, in=90] (-1.5,0) to [out=270,in=180] (0,-1.5);
\draw [very thick] (0,2.25) to [out=0,in=225](2.4,1.8);
\draw [very thick] (0,2.25) to [out=180,in=315](-2.4,1.8);
\node at (0,-1) {$*$};
\node at (0,-3) {$*$};
\draw [dashed] (0,-1) to (0,-3);
 \end{tikzpicture}
 \right)
   \end{align*}
  The projectors $T_{+^{m}}$ and $T_{-^{m}}$ for $m\geq 3$ are obtained via recursions of type~\eqref{eq:recursive} with base cases $T_{++}$ and $T_{--}$.
  \end{proposition}
\begin{proof}
The proof is again based on the application of the functor $\phi$.
\end{proof}


%

%

\begin{thebibliography}{RTW32}

\bibitem[APS04]{APS}
M.~M.~Asaeda, J.~H.~Przytycki, and A.~S.~Sikora.
\newblock Categorification of the {K}auffman bracket skein module of {I}--bundles
  over surfaces.
\newblock {\em Algebraic \& Geometric Topology}, 4(2):1177--1210, 2004.
\newblock \href{https://arxiv.org/abs/math/0409414}{arXiv:math/0409414}.

\bibitem[BN05]{BN2}
D.~Bar-Natan.
\newblock Khovanov's homology for tangles and cobordisms.
\newblock {\em Geom. Topol.}, 9:1443--1499, 2005.
\newblock \href{http://arxiv.org/abs/math/0410495}{arXiv:math/0410495}.

\bibitem[BW16]{BW1}
F.~Bonahon and H.~Wong.
\newblock Representations of the {K}auffman skein algebra {I}: invariants and
  miraculous cancellations.
\newblock {\em Inventiones mathematicae}, 204(1):195--243, 2016.
\newblock \href{http://arxiv.org/abs/1206.1638}{arXiv:1206.1638}.

\bibitem[EG99]{EG}
K.~Erdmann and R.~M.~Green.
\newblock On representations of affine {T}emperley-{L}ieb algebras. {II}.
\newblock {\em Pacific J. Math.}, 191(2):243--273, 1999.
\newblock \href{https://arxiv.org/abs/math/9811017}{arXiv:math/9811017}.

\bibitem[FG99]{FGr}
C.~K.~Fan and R.~M.~Green.
\newblock On the affine {T}emperley-{L}ieb algebras.
\newblock {\em J. London Math. Soc. (2)}, 60(2):366--380, 1999.
\newblock \href{https://arxiv.org/abs/q-alg/9706003}{arXiv:q-alg/9706003}.

\bibitem[FG00]{FG}
C.~Frohman and R.~Gelca.
\newblock Skein modules and the noncommutative torus.
\newblock {\em Trans. Amer. Math. Soc.}, 352(10):4877--4888, 2000.
\newblock \href{http://arxiv.org/abs/math/9806107}{arXiv:math/9806107}.

\bibitem[GL98]{GL}
J.~J.~Graham and G.~I.~Lehrer.
\newblock The representation theory of affine {T}emperley-{L}ieb algebras.
\newblock {\em Enseign. Math. (2)}, 44(3-4):173--218, 1998.

\bibitem[Gre98]{Gre}
R.~M.~Green.
\newblock On representations of affine {T}emperley-{L}ieb algebras.
\newblock In {\em Algebras and modules, {II} ({G}eiranger, 1996)}, volume~24 of
  {\em CMS Conf. Proc.}, pages 245--261. Amer. Math. Soc., Providence, RI,
  1998.

\bibitem[Jon85]{Jones}
V.~F.~R.~Jones.
\newblock A polynomial invariant for knots via von {N}eumann algebras.
\newblock {\em Bull. Amer. Math. Soc. (N.S.)}, 12(1):103--111, 1985.

\bibitem[Jon94]{Jones2}
V.~F.~R.~Jones.
\newblock A quotient of the affine {H}ecke algebra in the {B}rauer algebra.
\newblock {\em Enseign. Math. (2)}, 40(3-4):313--344, 1994.

\bibitem[JR06]{JR}
V.~F.~R.~Jones and S.~A.~Reznikoff.
\newblock Hilbert space representations of the annular {T}emperley-{L}ieb
  algebra.
\newblock {\em Pacific J. Math.}, 228(2):219--249, 2006.

\bibitem[Kau90]{KauffBracket}
L.H.~Kauffman.
\newblock State models and the {J}ones polynomial.
\newblock {\em New developments in the theory of knots}, 11:162, 1990.

\bibitem[Kho00]{Kh1}
M.~Khovanov.
\newblock A categorification of the {J}ones polynomial.
\newblock {\em Duke Math. J.}, 101(3):359--426, 2000.
\newblock \href{http://arxiv.org/abs/math/9908171}{arXiv:math/9908171}.

\bibitem[KS15]{KSa}
M.~Khovanov and R.~Sazdanovic.
\newblock Categorification of the polynomial ring.
\newblock {\em Fundamenta Mathematicae}, 230(3):251--280, 2015.
\newblock \href{http://arxiv.org/abs/1101.0293}{arXiv:1101.0293}.

\bibitem[L{\^e}16a]{Le16a}
T.~T.~Q.~L{\^e}.
\newblock On positivity of kauffman bracket skein algebras of surfaces.
\newblock {\em to appear in IMRN}, 2016.
\newblock \href{http://arxiv.org/abs/1603.08265}{arXiv:1603.08265}.

\bibitem[L{\^e}16b]{Le16b}
T.~T.~Q.~L{\^e}.
\newblock Triangular decomposition of skein algebras.
\newblock 2016.
\newblock \href{http://arxiv.org/abs/1609.04987}{arXiv:1609.04987}.

\bibitem[MS14]{MSa}
H.~{Morton} and P.~{Samuelson}.
\newblock {The HOMFLYPT skein algebra of the torus and the elliptic Hall
  algebra}.
\newblock 2014.
\newblock \href{http://arxiv.org/abs/1410.0859}{arXiv:1410.0859}.

\bibitem[Prz91]{Prz1}
J.~H.~Przytycki.
\newblock Skein modules of {$3$}-manifolds.
\newblock {\em Bull. Polish Acad. Sci. Math.}, 39(1-2):91--100, 1991.
\newblock \href{https://arxiv.org/abs/math/0611797}{arXiv:math/0611797}.

\bibitem[QR15]{QRu}
H.~Queffelec and H.~M.~Russell.
\newblock Chebyshev polynomials and the frohman--gelca formula.
\newblock {\em Journal of Knot Theory and Its Ramifications}, 24(04):1550023,
  2015.
\newblock \href{http://arxiv.org/abs/1403.3716}{arXiv:1403.3716}.

\bibitem[QW]{QW}
H.~Queffelec and P.~Wedrich.
\newblock A categorification of the {F}rohman-{G}elca formula.
\newblock in preparation.

\bibitem[RTW32]{RTW}
G.~Rumer, E.~Teller, and H.~Weyl.
\newblock Eine f\"ur die {V}alenztheorie geeignete {B}asis der bin\"aren
  {V}ektorinvarianten.
\newblock {\em Nachrichten von der Ges. der Wiss. Zu G\"ottingen. Math.-Phys.
  Klasse}, pages 498--504, 1932.
\newblock In German.

\bibitem[Thu14]{Thu}
D.~P.~Thurston.
\newblock Positive basis for surface skein algebras.
\newblock {\em Proc. Natl. Acad. Sci. USA}, 111(27):9725--9732, 2014.
\newblock \href{https://arxiv.org/abs/1310.1959}{arXiv:1310.1959}.

\bibitem[Tur91]{Tur}
V.~G.~Turaev.
\newblock Skein quantization of {P}oisson algebras of loops on surfaces.
\newblock {\em Ann. Sci. \'Ecole Norm. Sup. (4)}, 24(6):635--704, 1991.

\bibitem[Wen87]{Wen}
H.~Wenzl.
\newblock On sequences of projections.
\newblock {\em C. R. Math. Rep. Acad. Sci. Canada}, 9(1):5--9, 1987.

\end{thebibliography}
\end{document}